\documentclass[a4paper]{amsart}
\usepackage{amssymb, enumerate}

\usepackage{xy} \xyoption{all}

\usepackage{aliascnt}
\usepackage{hyperref}
\usepackage{graphicx}

\newcommand{\subseteqRotated}{\mathrel{\reflectbox{\rotatebox[origin=c]{270}{$\subseteq$}}}}
\newcommand{\subseteqRotatedUp}{\mathrel{\reflectbox{\rotatebox[origin=c]{90}{$\subseteq$}}}}
\newcommand{\eqRotated}{\mathrel{\reflectbox{\rotatebox[origin=c]{270}{$=$}}}}

\DeclareMathOperator{\locdim}{locdim}
\DeclareMathOperator{\Hom}{Hom}
\DeclareMathOperator{\Sep}{Sep}
\DeclareMathOperator{\rr}{rr}
\DeclareMathOperator{\xrr}{xrr}
\DeclareMathOperator{\Ext}{Ext}

\newcommand{\In}[0]{$\mathrm{I}_n$}
\newcommand{\Iinf}[0]{$\mathrm{I}_\infty$}
\newcommand{\IIone}[0]{$\mathrm{II}_1$}
\newcommand{\IIinf}[0]{$\mathrm{II}_\infty$}
\newcommand{\III}[0]{$\mathrm{III}$}
\newcommand{\andSep}{\,\,\,\text{ and }\,\,\,}
\newcommand{\Bdd}{\mathcal{B}}
\newcommand{\Cpct}{\mathcal{K}}
\newcommand{\Calk}{Q}
\newcommand{\tensMax}[0]{\otimes_{\mathrm{max}}}
\newcommand{\ZZ}{{\mathbb{Z}}}
\newcommand{\sa}{\text{sa}}
\newcommand{\NN}{{\mathbb{N}}}
\newcommand{\ca}{$C^*$-algebra}

\numberwithin{equation}{section}

\newtheorem{lma}{Lemma}[section]

\newaliascnt{thmCt}{lma}
\newtheorem{thm}[thmCt]{Theorem}
\aliascntresetthe{thmCt}

\newaliascnt{corCt}{lma}
\newtheorem{cor}[corCt]{Corollary}
\aliascntresetthe{corCt}

\newaliascnt{prpCt}{lma}
\newtheorem{prp}[prpCt]{Proposition}
\aliascntresetthe{prpCt}

\theoremstyle{definition}

\newaliascnt{pgrCt}{lma}
\newtheorem{pgr}[pgrCt]{}
\aliascntresetthe{pgrCt}

\newaliascnt{dfnCt}{lma}
\newtheorem{dfn}[dfnCt]{Definition}
\aliascntresetthe{dfnCt}

\newaliascnt{rmkCt}{lma}
\newtheorem{rmk}[rmkCt]{Remark}
\aliascntresetthe{rmkCt}

\newaliascnt{qstCt}{lma}
\newtheorem{qst}[qstCt]{Question}
\aliascntresetthe{qstCt}

\newaliascnt{exaCt}{lma}
\newtheorem{exa}[exaCt]{Example}
\aliascntresetthe{exaCt}

\newcounter{theoremintro}

\newaliascnt{thmIntroCt}{theoremintro}
\newtheorem{thmIntro}[thmIntroCt]{Theorem}
\aliascntresetthe{thmIntroCt}

\newaliascnt{exaIntroCt}{theoremintro}
\newtheorem{exaIntro}[exaIntroCt]{Example}
\aliascntresetthe{exaIntroCt}

\newaliascnt{qstIntroCt}{theoremintro}
\newtheorem{qstIntro}[qstIntroCt]{Question}
\aliascntresetthe{qstIntroCt}

\newaliascnt{pbmIntroCt}{theoremintro}
\newtheorem{pbmIntro}[pbmIntroCt]{Problem}
\aliascntresetthe{pbmIntroCt}

\title{Real rank of some multiplier algebras}

\author{Hannes Thiel}
\address{Hannes~Thiel, 
Department of Mathematical Sciences, Chalmers University of Technology and University of
Gothenburg, Gothenburg SE-412 96, Sweden.}
\email{hannes.thiel@chalmers.se}
\urladdr{www.hannesthiel.org}

\subjclass[2010]%
{Primary
46L05, % General theory of C*-algebras
46L85; % Noncommutative topology
Secondary
46L80, % K-theory and operator algebras (including cyclic theory) 
46M20. % Methods of algebraic topology in functional analysis
}


\thanks{
The author was partially supported by the Knut and Alice Wallenberg Foundation (KAW 2021.0140).
}

\date{\today}

\begin{document}

%==========================================================================================
\begin{abstract}
We show that there exists a separable, nuclear \ca{} with real rank zero and trivial $K$-theory such that its multiplier and corona algebra have real rank one.
This disproves two conjectures of Brown and Pedersen.

We also compute the real rank of the stable multiplier algebra and the stable corona algebra of countably decomposable type~\Iinf{} and type~\IIinf{} factors. 
Together with results of Zhang this completes the computation of the real rank for stable multiplier and corona algebras of countably decomposable factors.
\end{abstract}

\maketitle

%==========================================================================================
%==========================================================================================
\section{Introduction}

%==========================================================================================
The real rank is a noncommutative dimension theory that was introduced by Brown and Pedersen in \cite{BroPed91CAlgRR0}.
It assigns a number $\rr(A) \in \{0,1,\ldots,\infty\}$ to each \ca{} $A$, and it is considered a noncommutative generalization of covering dimension since for a compact, Hausdorff space $X$, the real rank of $C(X)$ agrees with $\dim(X)$, the covering dimension of $X$.

More accurately, the real rank should be thought of as a noncommutative generalization of \emph{local covering dimension}, since for a locally compact, Hausdorff space~$X$, the real rank of $C_0(X)$ agrees with $\locdim(X)$, which is equal to the supremum of the covering dimension of all compact subsets of $X$;
see \cite[Section~2.2(ii)]{BroPed09Limits}.
For $\sigma$-compact, locally compact, Hausdorff spaces, the local covering dimension agrees with the covering dimension (\cite[Proposition~5.3.4]{Pea75DimThy}), but in general the covering dimension can be strictly larger than the local covering dimension.

A major theme of research about noncommutative dimension theories in general, and the real rank in particular, is to determine to what extend results about the covering dimension of topological spaces can be generalized to the noncommutative setting.
In this paper, we consider the noncommutative analog of the classical result that for a $\sigma$-compact, locally compact, Hausdorff space $X$, the (local) covering dimension of $X$ agrees with that of its Stone-\v{C}ech compactification $\beta X$;
see for example %\cite[Theorem~9.5, p.48]{Nag70DimThy}, or \cite[Theorem~3.1.25]{Eng78DimThy}, or 
\cite[Proposition~6.4.3]{Pea75DimThy}.

Considering a \ca{} $A$ as a noncommutative topological space, the multiplier algebra $M(A)$ is the noncommutative analog of the Stone-\v{C}ech compactification.
One might therefore expect that the real rank of a \ca{} $A$ agrees with that of its multiplier algebra, at least when $A$ is $\sigma$-unital (the noncommutative analog of $\sigma$-compactness).
One always has $\rr(A) \leq \rr(M(A))$, but in general the real rank of $A$ can be strictly smaller than that of $M(A)$, already for separable, subhomogeneous \ca{s};
see \cite[Example~3.16(i)]{Bro16HigherRRSR}.
This raises the problem of finding necessary and sufficient conditions for the equality $\rr(A) = \rr(M(A))$.

A particularly relevant instance of this problem occurs for $\rr(A)=0$, in which case we aim to determine when the multiplier algebra $M(A)$ has real rank zero.
This problem has interesting connections to generalizations of the Weyl-von Neumann theorem and quasidiagonality, which has been extensively studied by Lin \cite{Lin91GenWeylVN1, Lin93ExpRank, Lin95GenWeylVN2} and Zhang \cite{Zha91QuasidiagInterpolMultProj, Zha92RR0CoronaMultiplier1, Zha92RR0CoronaMultiplier2, Zha90RR0CoronaMultiplier3, Zha92RR0CoronaMultiplier4}.

First, we see that $\rr(A)=0$ does not generally imply $\rr(M(A))=0$ by considering the stabilization of the Calkin algebra (\autoref{exa:StableMultCalkin}), for which we have
\[
0 = \rr(Q\otimes\Cpct) < \rr(M(Q\otimes\Cpct)) = 1.
\]
More generally, it is known that a $K$-theoretic obstruction arises:
If $\rr(M(A))=0$ and if the $K$-theory of $M(A)$ vanishes (which is for example automatic if $A$ is stable), then $K_1(A)=0$.
This includes the above example, since $K_1(Q\otimes\Cpct)\cong\ZZ \neq 0$.

This lead Brown and Pedersen to formulate the following three conjectures in \cite[Remarks~3.22]{BroPed91CAlgRR0}: 
\begin{enumerate}
\item
$\rr(M(A))=0$ for every AF-algebra $A$;
\item
$\rr(M(A))=0$ for every $A$ with $\rr(A)=0$ and $K_1(A)=0$;
\item
$\rr(M(A)/A)=0$ for every $A$ with $\rr(A)=0$.
\end{enumerate}

The initial evidence for the first conjecture was a verification for matroid algebras (a special class of AF-algebras) by Brown-Pedersen \cite[Theorem~3.21]{BroPed91CAlgRR0} and Higson-R{\o}rdam \cite[Theorem~4.4]{HigRor91WeylVNMultSomeAF}.
Soon after, Lin showed that $\rr(M(A))=0$ whenever $A$ is a $\sigma$-unital \ca{} with real rank zero, stable rank one and $K_1(A)=0$;
see \cite[Theorem~10]{Lin93ExpRank}.
Since AF-algebras have these properties, this in particular confirmed the first conjecture.
It also verified the second conjecture under the additional assumption of stable rank one (and $\sigma$-unitality).

The third conjecture has been confirmed by Lin for $\sigma$-unital, simple \ca{s} with real rank zero and stable rank one (\cite[Theorem~15]{Lin93ExpRank}), and by Zhang for $\sigma$-unital, simple, purely infinite \ca{s} (\cite[Corollary 2.6(i)]{Zha92RR0CoronaMultiplier1}).
Note that the second and third conjectures are closely related:
By considering the extension
\[
0 \to A \to M(A) \to M(A)/A \to 0
\]
we see that $\rr(M(A))=0$ if and only if $\rr(A)=\rr(M(A)/A)=0$ and the index map $K_0(M(A)/A) \to K_1(A)$ vanishes;
see \cite[Proposition~4]{LinRor95ExtLimitCircle}.
This shows that the third conjecture is stronger than the second, and that both conjectures are equivalent if $A$ is stable.

\medskip

In this paper, we settle the second and third Brown-Pedersen conjectures negatively by exhibiting several counterexamples.
This also solves Problem~13 in \cite{Zha07OpenProbMult}.

%==========================================================================================
\begin{exaIntro}[\ref{exa:SeparableBPCounterexample}]
There exists a separable, nuclear \ca{} $A$ with real rank zero, with trivial $K$-theory and such that
\[
\rr(M(A)) = \rr(M(A)/A) = 1.
\]
\end{exaIntro}

%==========================================================================================
Another unexpectedly easy counterexample is given by the countably decomposable type~\Iinf{} factor, that is, the algebra of bounded operators on a separable, infinite-dimensional Hilbert space.
Further examples are given by countably decomposable type~\IIinf{} factors.
In both cases, the stable multiplier and corona algebra have real rank one.
Together with Zhang's results for type~\IIone{} factors (\cite[Example~2.11]{Zha92RR0CoronaMultiplier1}) and for countably decomposable type~\III~factors (\cite[Examples~2.7(iv)]{Zha92RR0CoronaMultiplier1}, see also \autoref{exa:StableMultIII}), this completes the computation of the real rank for stable multiplier and corona algebras of countably decomposable factors:

%==========================================================================================
\begin{thmIntro}[{\ref{prp:rrStableMultBdd}}, \ref{prp:rrStableMultIIinf}]
\label{prp:rrMultFactor}
Let $N$ be a countably decomposable factor.
%Then $N\otimes\Cpct$ is a $\sigma$-unital \ca{} with real rank zero and $K_1(N\otimes\Cpct)=0$.

If $N$ has type~\Iinf{} or type~\IIinf{}, then 
\[
\rr( M(N\otimes\Cpct) ) 
= \rr( M(N\otimes\Cpct)/(N\otimes\Cpct) ) 
= 1.
\]

If $N$ has type~\IIone{} or type~\III, then
\[
\rr( M(N\otimes\Cpct) ) 
= \rr( M(N\otimes\Cpct)/(N\otimes\Cpct) ) 
= 0.
\]
\end{thmIntro}

%==========================================================================================
This raises the problem of computing the real rank for stable multiplier and corona algebras of factors that are not countably decomposable, and more generally of arbitrary von Neumann algberas;
see \autoref{rmk:VNA}.

%==========================================================================================
We ask if there are also simple counterexamples to the second and third Brown-Pedersen conjectures:

%==========================================================================================
\begin{qstIntro}
Is there a \emph{simple}, separable \ca{} $A$ with $\rr(A)=0$ and $K_1(A)=0$ such that $\rr(M(A)/A) \neq 0$?
\end{qstIntro}

%==========================================================================================
\subsection*{Methods}

%==========================================================================================
One difficulty in computing the real rank of the stable multiplier algebras of a factor $N$ of type~\Iinf{} or type~\IIinf{} is that $N$ is not simple as a \ca{}.
To address this, we study the following problem:
Given an extension
\[
0 \to A \to E \to B \to 0
\]
of \ca{s}, how can we estimate the real rank of $M(E)$?
It is natural to assume that $E$ is $\sigma$-unital, since then the map $E \to B$ naturally induces a surjective morphism $M(E) \to M(B)$, and we obtain an extension
\[
0\to J \to M(E) \to M(B) \to 0
\]
where $J$ is a hereditary sub-\ca{} in $M(A)$;
see \autoref{pgr:MultExt} and \autoref{prp:KernelMultExt}.

In \cite{Thi23arX:RRExt}, the \emph{extension real rank}, denoted by $\xrr(\cdot)$, was introduced as a method to bound the real rank of an extension of \ca{s}.
(The main results are recalled in \autoref{sec:xrr}.)
Applying \autoref{prp:EstimateWithXRR}, we obtain
\[
\max\big\{ \rr(J), \rr(M(B)) \big\} 
\leq \rr(M(E)) 
\leq \max\big\{ \xrr(J), \rr(M(B)) \big\}.
\]

We are thus faced with the problem of computing the (extension) real rank of hereditary sub-\ca{s} of $M(A)$.

%==========================================================================================
\begin{pbmIntro}
\label{pbm:xrrHerMA}
Given a \ca{} $A$, estimate the (extension) real rank of hereditary sub-\ca{s} of $M(A)$.
\end{pbmIntro}

%==========================================================================================
In Sections~\ref{sec:ExtSimplePI} and~\ref{sec:ExtSimpleSR1}, we solve this problem for simple, purely infinite \ca{s}, and for certain simple \ca{s} with stable rank one.
The latter includes as a special case $A=\Cpct$, which we consider in \autoref{sec:ExtCpct}.
As a consequence, we obtain estimates for the real rank of multiplier algebras of extensions by such ideals:

%==========================================================================================
\begin{thmIntro}[\ref{prp:rrMultExtByKK}, \ref{prp:rrMultExtSimplePI}, \ref{prp:rrMultExtSimpleSR1}]
Let $0 \to A \to E \to B \to 0$ be an extension of \ca{s}.
Assume that $A$ is simple, and that $E$ is $\sigma$-unital.
Assume further, that $A$ is purely infinite, or that $A$ has real rank zero, stable rank one, strict comparison of positive elements by traces, and finitely many extremal traces (normalized at some nonzero projection).
Then:
\[
\rr(M(B)) \leq \rr(M(E)) \leq \max\big\{1, \rr(M(B)) \big\}.
\]
If we additionally assume that $K_1(A)=0$, then
\[
\rr(M(E)) = \rr(M(B)).
\]
\end{thmIntro}

%==========================================================================================
%==========================================================================================
\subsection*{Conventions}

%==========================================================================================
Throughout, we let $\Cpct$ denote the \ca{} of compact operators on a separable, infinite-dimensional Hilbert space.
Further, we let $\Bdd$ denote the algebra of bounded operators on a separable, infinite-dimensional Hilbert space.
The Calkin algebra is $\Calk := \Bdd/\Cpct$.
Given a \ca{} $A$, we let $A_\sa$ and $A_+$ denote the subsets of self-adjoint and positive elements, respectively.

By an extension $0 \to A \to E \to B \to 0$ of \ca{s}, we mean that $E$ is a \ca{} containing $A$ as a closed ideal such that $B$ is isomorphic to $E/A$.

%==========================================================================================
\subsection*{Acknowledgements}

%==========================================================================================
The author thanks Eduard Vilalta for valuable feedback on a first version of this paper.

%==========================================================================================
%==========================================================================================
\section{Preliminaries on real rank of extensions}
\label{sec:xrr}

%==========================================================================================
In this section, we recall estimates of the real rank of an extension of \ca{s} from \cite{Thi23arX:RRExt}.
We begin with the definition of the (extension) real rank.

Let $A$ be a unital \ca.
For $n \geq 1$, a self-adjoint tuple $(a_1,\ldots,a_n) \in A_\sa^n$ is said to be \emph{unimodular} if it generate $A$ as a left ideal.
One can show that a tuple $(a_1,\ldots,a_n) \in A_\sa^n$ is unimodular if and only if $\sum_{k=1}^n a_k^2$ is invertible.

%==========================================================================================
\begin{dfn}[{Brown-Pedersen \cite{BroPed91CAlgRR0}}]
The \emph{real rank} of a unital \ca{}~$A$, denoted by $\rr(A)$, is defined as the smallest integer $n \geq 0$ such that the set of unimodular, self-adjoint $(n+1)$-tuples is dense in $A_\sa^{n+1}$.
By convention, we set $\rr(A)=\infty$ if there is no $n$ with this property.
The real rank of a nonunital \ca{} is defined as that of its minimal unitization. 
%$\rr(A) := \rr(\widetilde{A})$, 
\end{dfn}

%==========================================================================================
\begin{dfn}[{\cite[Defintion~3.1]{Thi23arX:RRExt}}]
\label{dfn:xrr}
Let $A$ be a \ca{}, let $n \geq 0$, and let $\pi_A \colon M(A) \to M(A)/A$ denote the quotient map.
We say that $A$ has property $(\Lambda_n)$ if for every $(a_0,\ldots,a_n) \in M(A)_\sa^{n+1}$ such that $(\pi_A(a_0),\ldots,\pi_A(a_n))$ is unimodular, and for every $\varepsilon>0$, there exists a unimodular tuple $(b_0,\ldots,b_n) \in M(A)_\sa^{n+1}$ such that
\[
\| a_0-b_0 \| < \varepsilon, \ldots, \| a_n-b_n \| < \varepsilon, \andSep
\pi_A(a_0)=\pi_A(b_0), \ldots,
\pi_A(a_n)=\pi_A(b_n).
\]

The \emph{extension real rank} of $A$, denoted by $\xrr(A)$, is the smallest integer $n \geq 0$ such that $A$ has property $(\Lambda_m)$ for all $m \geq n$.
We set $\xrr(A)=\infty$ if there is no $n$ such that $A$ has $(\Lambda_m)$ for all $m \geq n$.
\end{dfn}

%==========================================================================================
\begin{thm}[{\cite[Theorem~2.4]{Thi23arX:RRExt}}]
\label{prp:EstimateBasic}
Let $0 \to A \to E \to B \to 0$ be an extension of \ca{s}.
Then
\[
\max\{ \rr(A), \rr(B) \}
\leq \rr(E)
\leq \rr(A) + \rr(B) + 1.
\]
\end{thm}

%==========================================================================================
\begin{thm}[{\cite[Theorem~3.7]{Thi23arX:RRExt}}]
\label{prp:EstimateWithXRR}
Let $0 \to A \to E \to B \to 0$ be an extension of \ca{s}.
Then
\[
\max\{ \rr(A), \rr(B) \}
\leq \rr(E)
\leq \max\big\{ \xrr(A), \rr(B) \big\}.
\]
\end{thm}

%==========================================================================================
The next result allows us to estimate the extension real rank of a \ca{} given as an extension.

%==========================================================================================
\begin{prp}[{\cite[Proposition~4.2]{Thi23arX:RRExt}}]
\label{prp:xrrOfExtension}
Let $0 \to A \to E \to B \to 0$ be an extension of \ca{s}.
Then
\[
\xrr(E)
\leq \max\big\{ \xrr(A), \xrr(B) \big\}.
\]
\end{prp}

%==========================================================================================
\begin{prp}[{\cite[Corollary~5.4]{Thi23arX:RRExt}}]
\label{prp:xrr0SufficientCond}
Let $A$ be a \ca{}.
Assume that $\xrr(A) \leq 1$, $\rr(A)=0$ and $K_1(A)=0$.
Then $\xrr(A)=0$.
\end{prp}

%==========================================================================================
%==========================================================================================
\section{Real rank of multiplier algebras of some simple C*-algebras} 
\label{sec:MultSimple}

%==========================================================================================
In this section, we estimate the real rank of the multiplier algebra of some simple \ca{s}.
In particular, and based on the seminal work of Zhang and Lin, we compute the real rank of the multiplier algebra of $\sigma$-unital, simple \ca{s} that are either purely infinite, or have real rank zero and stable rank one;
see \autoref{prp:rrMultSimple}.
We recover the computation of the real rank of the stable multiplier algebra of \IIone{} factors (\autoref{exa:StableMultII-1}) and of (countably decomposable) type~\III{} factors (\autoref{exa:StableMultIII}).

\medskip

In the next result, we consider the index map $K_0(M(A)/A) \to K_1(A)$ from the six-term exact sequence in $K$-theory (\cite[Corollary~V.1.2.22]{Bla06OpAlgs}) induced by the extension $0 \to A \to M(A) \to M(A)/A \to 0$.
If $A$ and $M(A)/A$ have real rank zero, then the vanishing of this index map is equivalent to the condition that every projection from $M(A)/A$ lifts to $M(A)$;
see, for example, \cite[Proposition~2.5]{Thi23arX:RRExt}.
If, moreover, $A$ is stable, then the $K$-theory of $M(A)$ vanishes (\cite[Theorem~10.2]{Weg93KThyBook}), and then it is also equivalent to the condition $K_1(A)=0$.

In \autoref{prp:rrMultSimple} below we will see that for many simple \ca{s} with real rank zero, the corona algebra has real rank zero as well, which then allows us to compute the real rank of the multiplier algebra.
In \cite{LinNg16CoronaStableJiangSu} and \cite{Ng22RR0PICorona} it is shown that the corona algebra also has real rank zero for certain simple \ca{s} with nonzero real rank, such as the Jiang-Su algebra.

%==========================================================================================
\begin{prp}
\label{prp:rrMult-rr0Corona}
Let $A$ be a \ca{} such that the corona algebra $M(A)/A$ has real rank zero.
Then
\[
\rr(A) \leq \xrr(A) = \rr(M(A)) \leq \rr(A)+1.
\]

If $A$ and $M(A)/A$ have real rank zero, then
\[
\xrr(A)
= \rr(M(A))
= \begin{cases}
0, &\text{if the index map $K_0(M(A)/A) \to K_1(A)$ vanishes} \\
1, &\text{otherwiese}
\end{cases}.
\]
\end{prp}
\begin{proof}
Applying Theorems~\ref{prp:EstimateBasic} and~\ref{prp:EstimateWithXRR} for the extension
\[
0 \to A \to M(A) \to M(A)/A \to 0,
\] 
we get
\begin{align*}
\rr(M(A))
&\leq \rr(A)+\rr(M(A)/A)+1
= \rr(A)+1, \andSep \\
\rr(M(A))
&\leq \max\{ \xrr(A), \rr(M(A)/A) \big\}
= \xrr(A).
\end{align*}
By \cite[Proposition~3.11]{Thi23arX:RRExt}, we have
\[
\rr(A)
\leq \xrr(A) 
\leq \rr(M(A)).
\]
By combining these estimates, we get $\rr(A) \leq \xrr(A) = \rr(M(A)) \leq \rr(A)+1$.

By \cite[Corollary~2.6]{Thi23arX:RRExt}, an extension of real rank zero \ca{s} has real rank zero or one, depending on whether the index map from $K_0$ of the quotient to~$K_1$ of the ideal vanishes.
The computation of $\rr(M(A))$ follows.
\end{proof}

%==========================================================================================
\begin{thm}
\label{prp:rrMultSimple}
Let $A$ be a $\sigma$-unital, simple \ca{} that is purely infinite or has real rank zero and stable rank one.
Then $\rr(M(A)/A)=0$ and
\[
\xrr(A)
= \rr(M(A))
= \begin{cases}
0, &\text{if the index map $K_0(M(A)/A) \to K_1(A)$ vanishes} \\
1, &\text{otherwise}
\end{cases}.
\]
\end{thm}
\begin{proof}
We will see that in both cases we have $\rr(A)=0$ and $\rr(M(A)/A)=0$, whence the result follows from \autoref{prp:rrMult-rr0Corona}.
In the case that $A$ has real rank zero and stable rank one, we have $\rr(M(A)/A)=0$ by \cite[Theorem~15]{Lin93ExpRank}.
Now assume that $A$ is purely infinite.
Then $\rr(A)=0$ by \cite[Proposition~V.3.2.12]{Bla06OpAlgs};
see also \cite[Theorem~1.2(ii)]{Zha92RR0CoronaMultiplier1}.
Further, $A$ is either unital or stable by \cite[Theorem~1.2(i)]{Zha92RR0CoronaMultiplier1}.
It follows that $\rr(M(A)/A)=0$, using \cite[Corollary~2.6(i)]{Zha92RR0CoronaMultiplier1} in the stable case.
\end{proof}

%==========================================================================================
\begin{cor}
\label{prp:rrMultStableSimple}
Let $A$ be a unital, simple \ca{} that is purely infinite or has real rank zero and stable rank one.
Then $\rr(M(A\otimes\Cpct)/(A\otimes\Cpct) = 0$ and 
\[
\xrr(A\otimes\Cpct)
= \rr(M(A\otimes\Cpct))
= \begin{cases}
0, &\text{if $K_1(A)=0$} \\
1, &\text{otherwise}
\end{cases}.
\]
\end{cor}
\begin{proof}
By \cite[Theorem~10.2]{Weg93KThyBook}, the $K$-theory of the the multiplier algebra of a stable \ca{} vanishes.
Applying the six-term exact sequence in $K$-theory (\cite[Corollary~V.1.2.22]{Bla06OpAlgs}), it follows that the index map
\[
K_0(M(A\otimes\Cpct)/(A\otimes\Cpct)) \to K_1(A\otimes\Cpct) \cong K_1(A)
\]
is an isomorphism.
Thus, the index map vanishes if and only if $K_1(A)=0$.
Now the result follows from \autoref{prp:rrMultSimple}.
\end{proof}

%==========================================================================================
\begin{exa}
\label{exa:StableMultCalkin}
Consider the Calkin algebra $\Calk := \Bdd/\Cpct$.
Then $\Calk$ is a unital, simple, purely infinite \ca{} with $K_1(\Calk) \cong \ZZ \neq 0$.
It follows from \autoref{prp:rrMultStableSimple} that
\[
\rr( M(\Calk\otimes\Cpct)/(\Calk\otimes\Cpct) ) 
= 0, \andSep
\xrr(\Calk\otimes\Cpct) 
= \rr( M(\Calk\otimes\Cpct) ) 
= 1.
\]

The computation of the real rank of the stable corona algebra of $\Calk$ was first obtained by Zhang \cite[Examples~2.7(iii)]{Zha92RR0CoronaMultiplier1}.
\end{exa}

%==========================================================================================
\begin{exa}
\label{exa:StableMultIrratRot}
Let $A_\theta$ be an irrational rotation algebra.
We refer to \cite[Examples~II.8.3.3(i) and~II.10.4.12(i)]{Bla06OpAlgs} for details.
The \ca{} $A_\theta$ is unital, simple, has real rank zero (\cite[Theorem~1.5]{BlaKumRor92ApproxCentralMatUnit}), stable rank one (\cite[Theorem~1]{Put90InvDenseIrratRot}), and $K_1(A_\theta) \cong \ZZ^2 \neq 0$ (\cite[Section~12.3]{Weg93KThyBook}).
It follows from \autoref{prp:rrMultStableSimple} that
\[
\rr( M(A_\theta\otimes\Cpct)/(A_\theta\otimes\Cpct) ) 
= 0, \andSep
\xrr(A_\theta\otimes\Cpct) 
= \rr(M(A_\theta\otimes\Cpct)) 
= 1.
\]

The computation of the real rank of the stable corona algebra of $A_\theta$ was first obtained by Lin \cite[Theorem~15]{Lin93ExpRank}.
\end{exa}

%==========================================================================================
We now turn to the stable multiplier algebras of von Neumann factors.
Using the description of the norm-closed ideals in factors (\cite[Proposition~III.1.7.11]{Bla06OpAlgs}), we see that a factor $N$ is simple as a \ca{} if and only if it is type~\In{} for finite~$n$ (that is, $N \cong M_n(\mathbb{C})$), $N$ is type~\IIone, or $N$ is countably decomposable and type~\III.
In the case of type~\In, we have $N\otimes\Cpct \cong \Cpct$, and then 
\[
\rr(M(N\otimes\Cpct))
= \rr(\Bdd)
= 0.
\]
We now consider the real rank of the stable multiplier algebra of factors of type~\IIone{} and type~\III.
It is well-known that every von Neumann algebra has real rank zero and vanishing $K_1$-group, and that every finite von Neumann algebra has stable rank one.

%==========================================================================================
\begin{exa}
\label{exa:StableMultII-1}
Let $N$ be a type~\IIone{} factor.
Then $N$ is a unital, simple \ca{} that has real rank zero, stable rank one, and vanishing $K_1$-group.
Thus, we have
\[
\rr(M(N\otimes\Cpct)) = 0, \andSep
\xrr(N \otimes \Cpct)
= \rr(M(N \otimes \Cpct))
= 0
\]
by \autoref{prp:rrMultStableSimple}.
The computation of the real rank of the stable multiplier algebra of type~\IIone{} factors was first obtained by Zhang \cite[Examples~2.11]{Zha92RR0CoronaMultiplier1};
see also \cite[Theorem~10]{Lin93ExpRank}.
\end{exa}

%==========================================================================================
\begin{exa}
\label{exa:StableMultIII}
Let $N$ be a countably decomposable type~\III{} factor.
Then $N$ is a unital, simple, purely infinite \ca{} with $K_1(N)=0$.
Thus, we have
\[
\xrr(N \otimes \Cpct)
= \rr(M(N \otimes \Cpct))
= 0
\]
by \autoref{prp:rrMultStableSimple}.
The computation of the real rank of $M(N \otimes \Cpct))$ was first obtained by Zhang \cite[Examples~2.7(iv)]{Zha92RR0CoronaMultiplier1}.
However, contrary to what is stated in \cite[Examples~2.7(iv)]{Zha92RR0CoronaMultiplier1}, a type~\III{} factor is only simple if it is countably decomposable.
Therefore, it remains open if we have $\rr(M\otimes\Cpct)=0$ for arbitrary type~\III{} factors;
see also \autoref{rmk:VNA}.
\end{exa}

%==========================================================================================
%==========================================================================================
%\section{The real rank of multiplier and corona algebras} 
\section{Extensions by the algebra of compact operators} 
\label{sec:ExtCpct}

%==========================================================================================
In this section, we consider extensions
\[
0 \to \Cpct \to E \to B \to 0
\]
where $E$ is $\sigma$-compact. %, and with $\Cpct$ denoting the algebra of compact operators on a separable, infinite-dimensional Hilbert space.
We solve \autoref{pbm:xrrHerMA} for $\Cpct$ (\autoref{prp:xrrHerBdd}) and deduce in \autoref{prp:rrMultExtByKK} that
\[
\rr(M(E)) = \rr(M(B)).
\]

We use this to exhibit a separable, nuclear counterexample to the second and third Brown-Pedersen conjectures (\autoref{exa:SeparableBPCounterexample}) and to compute the real rank of the stable multiplier and corona algebra of the algebra $\Bdd$ of bounded operators on a separable, infinite-dimensional Hilbert space (\autoref{prp:rrStableMultBdd}).

\medskip

First, we consider the general framework for studying the multiplier and corona algebras of extensions.

%==========================================================================================
\begin{pgr}
\label{pgr:MultExt}
Let $0 \to A \to E \xrightarrow{\pi} B \to 0$ be an extension of \ca{s}. %, and let $\pi \colon E \to B$ denote the quotient map.
The natural extension $\pi^{**}\colon E^{**} \to B^{**}$ is a surjective morphism whose kernel is naturally isomorphic to $A^{**}$;
see \cite[III.5.2.11]{Bla06OpAlgs}.
We realize the multiplier algebra of a \ca{} $D$ as
\[
M(D) = \big\{ x\in D^{**} : xD+Dx \subseteq D \big\}.
\]
Applying this for $E$ and $B$, we see that $\pi^{**}$ maps $M(E)$ to $M(B)$, and we denote this morphism by $\bar{\pi} \colon M(E) \to M(B)$. 

Let us now assume that $E$ is $\sigma$-unital.
Then $\bar{\pi}$ is surjective by \cite[Theorem~10]{Ped86SAW}, and we set $J:=\ker(\bar{\pi})$.
We obtain the following inclusions of extensions:
\[
\xymatrix@R-10pt{
0 \ar[r] & A \ar[r] & E \ar[r]^{\pi} & B \ar[r] & 0 \\
0 \ar[r] 
& J \ar[r] \ar@{}[u]|{\subseteqRotated} 
& M(E) \ar[r]^{\bar{\pi}} \ar@{}[u]|{\subseteqRotated} 
& M(B) \ar[r] \ar@{}[u]|{\subseteqRotated} & 0 \\
0 \ar[r] 
& A^{**} \ar[r] \ar@{}[u]|{\subseteqRotated}
& E^{**} \ar[r]^{\pi^{**}} \ar@{}[u]|{\subseteqRotated}
& B^{**} \ar[r] \ar@{}[u]|{\subseteqRotated} & 0 \\
}
\]

We thus have
\[
J = A^{**} \cap M(E) \subseteq A^{**}.
\]
%Next, we show that $J\subseteq M(I)\subseteq I^{**}$ as a hereditary sub-\ca.

The surjective morphism $\bar{\pi} \colon M(E) \to M(B)$ induces a surjective morphism $\bar{\bar{\pi}} \colon M(E)/E \to M(B)/B$, with kernel $K := \ker(\bar{\bar{\pi}})$.
We obtain the following extension:
\[
\xymatrix@R-10pt{
0 \ar[r] 
& K \ar[r] 
& M(E)/E \ar[r]^{\bar{\bar{\pi}}} 
& M(B)/B \ar[r] 
& 0
}
\]
Note that $K$ is naturally isomorphic to $J/A$.
\end{pgr}

%==========================================================================================
\begin{lma}
\label{prp:KernelMultExt}
Retain the notation from \autoref{pgr:MultExt}.
Then $A \subseteq J \subseteq M(A)$, and~$J$ is a hereditary sub-\ca{} of $M(A)$.
Further, $K=J/A \subseteq M(A)/A$ is a hereditary sub-\ca{} of $M(A)/A$.
\end{lma}
\begin{proof}
We will use that $J = A^{**} \cap M(E) \subseteq A^{**}$.
We first establish the following:

Claim: \emph{Let $x \in J$ and $e \in E$. Then $xe, ex \in A$.}
Indeed, since $x \in J \subseteq M(E)$ and $e \in E$, we have $xe,ex \in E$.
Further, since $x \in J \subseteq A^{**}$, and since $A^{**}$ is an ideal (in fact a summand) in $E^{**}$, we have $xe, ex \in A^{**}$.
Using that $A = E \cap A^{**}$, we obtain $xe,ex \in A$.
This proves the claim.

It follows directly from the claim that $J$ is contained in $M(A) = \{ x\in A^{**} : xA+Ax \subseteq A \}$.
To show that $J \subseteq M(A)$ is hereditary, let $x,y\in J$ and let $z \in M(A)$.
We need to verify $xzy \in J$.
We clearly have $xyz\in A^{**}$, and it remains to verify that $xzy \in M(E)$.
Let $e \in E$.
Using the claim, we have $ex \in A$.
Since $z \in M(A)$, we get $axz \in A$, and then $axzy \in A \subseteq E$.
Analogously, we have $xzye\in E$.
Since this holds for every $e \in E$, we have $xzy \in M(E)$, as desired.

Using that $J \subseteq M(A)$ is hereditary, it follows that $K =J/A \subseteq M(A)/A$ is hereditary as well.
\end{proof}

%==========================================================================================
\begin{prp}
\label{prp:rrMultGeneralExt}
Let $0 \to A \to E \to B \to 0$ be an extension of \ca{s}, and assume that $E$ is $\sigma$-unital.

If $\rr(M(A))=0$, then
\[
\rr(M(B)) \leq \rr(M(E)) \leq \rr(M(B)) + 1.
\]

If $\rr(M(A)/A)=0$, then
\[
\rr(M(B)/B) \leq \rr(M(E)/E) \leq \rr(M(B)/B) + 1.
\]
\end{prp}
\begin{proof}
By \autoref{pgr:MultExt} and \autoref{prp:KernelMultExt}, we have extensions
\[
0 \to J \to M(E) \to M(B) \to 0, \andSep
0 \to K \to M(E)/E \to M(B)/B \to 0,
\]
where $J$ is a hereditary sub-\ca{} of $M(A)$, and $K$ is a hereditary sub-\ca{} of $M(A)/A$.

By \cite[Corollary~2.8]{BroPed91CAlgRR0}, real rank zero passes to hereditary sub-\ca{s}.
Thus, if $\rr(M(A))=0$, then $\rr(J)=0$, and the estimate for $\rr(M(E))$ follows from \autoref{prp:EstimateBasic}.
Similarly, if $\rr(M(A)/A)=0$, then $\rr(K)=0$, and the estimate for $\rr(M(E)/E)$ follows analogously.
\end{proof}

%==========================================================================================
Thus, in the setting of \autoref{prp:rrMultGeneralExt}, if $\rr(M(A))=0$, then the real rank of $M(E)$ can take at most two values, and we are led to wonder if $\rr(M(E))$ is equal to $\rr(M(B))$ or to $\rr(M(B))+1$.
The following result of Brown-Pedersen provides an answer when $M(B)$ has real rank zero.
This suggests \autoref{qst:rr:MultExtension} below.

%==========================================================================================
\begin{prp}[{\cite[Theorem~4.8]{BroPed09Limits}}]
\label{prp:MultExtRR0}
Let $0 \to A \to E \to B \to 0$ be an extension of \ca{s}.
Assume that $E$ is $\sigma$-unital, and that $\rr(M(A))=0$ and $\rr(M(B))=0$.
Then $\rr(M(E)) = 0$.
\end{prp}

%==========================================================================================
\begin{qst}
\label{qst:rr:MultExtension}
Let $0 \to A \to E \to B \to 0$ be an extension of \ca{s}.
Assume that $E$ is $\sigma$-unital, and that $\rr(M(A))=0$.
Do we have $\rr(M(E)) = \rr(M(B))$?
\end{qst}

%==========================================================================================
In \autoref{prp:rrMultExtByKK}, we will answer \autoref{qst:rr:MultExtension} positively for the case $A=\Cpct$.
Recall that we let $Q:=\Bdd/\Cpct$ denote the Calkin algebra.
%Recall that we use $\Bdd$ to denote the algebra of bounded operators on a separable, infinite-dimensional Hilbert space, and we let $Q:=\Bdd/\Cpct$ denote the Calkin algebra.

%==========================================================================================
\begin{lma}
\label{prp:xrrHerBdd}
We have $\xrr(J) = 0$ for every hereditary sub-\ca{} $J \subseteq \Bdd$.
Further, we have $\xrr(K) \leq 1$ for every hereditary sub-\ca{} $K \subseteq Q$.
\end{lma}
\begin{proof}
If $K$ is a hereditary sub-\ca{} of $Q$, then $K$ is simple and purely infinite and therefore $\xrr(K) \leq 1$ by \cite[Proposition~5.11(2)]{Thi23arX:RRExt}.
Now, if $J$ is a hereditary sub-\ca{} of $\Bdd$, then we obtain an extension
\[
0 \to J\cap\Cpct \to J \to J/(J\cap\Cpct) \to 0.
\]
Note that $J\cap\Cpct$ is a hereditary sub-\ca{} of $\Cpct$ and therefore isomorphic to~$\Cpct$ or to a complex matrix algebra.
In either case, we have $\xrr(J\cap\Cpct)=0$, for example by \cite[Proposition~5.9(1)]{Thi23arX:RRExt}.
Further, $J/(J\cap\Cpct)$ is a hereditary sub-\ca{} of $Q$, and thus $\xrr(J/(J\cap\Cpct)) \leq 1$ as shown above.
Applying that the extension real rank does not increase when passing to extensions, \autoref{prp:xrrOfExtension}, we get
\[
\xrr(J) \leq \max\big\{ \xrr(J\cap\Cpct), \xrr(J/(J\cap\Cpct)) \big\} \leq 1.
\]
Since $\rr(\Bdd)=0$, and since real rank zero passes to hereditary sub-\ca{s} by \cite[Corollary~2.8]{BroPed91CAlgRR0}, we also have $\rr(J)=0$.
Further, we have $K_1(J)=0$ by \cite[Theorem~2.1]{Zha90TrivialK1Flow}.
Then $\xrr(J)=0$ by \autoref{prp:xrr0SufficientCond}.
\end{proof}

%==========================================================================================
\begin{prp}
\label{prp:rrMultExtByKK}
Let $0 \to \Cpct \to E \to B \to 0$ be an extension of \ca{s} and assume that $E$ is $\sigma$-unital.
Then
\[
\rr(M(E))
= \rr(M(B)),
\]
and
\[
\rr(M(B)/B) \leq \rr(M(E)/E) \leq \max\big\{1, \rr(M(B)/B) \big\}.
\]
\end{prp}
\begin{proof}
By \autoref{pgr:MultExt} and \autoref{prp:KernelMultExt}, we have extensions
\[
0 \to J \to M(E) \to M(B) \to 0, \andSep
0 \to K \to M(E)/E \to M(B)/B \to 0,
\] 
where $J$ is a hereditary sub-\ca{} of $\Bdd$, and $K$ is a hereditary sub-\ca{} of $Q$.
By \autoref{prp:xrrHerBdd}, we have $\xrr(J) = 0$ and $\xrr(K) \leq 1$, and now the result follows from \autoref{prp:EstimateWithXRR}.
\end{proof}

%==========================================================================================
\begin{exa}
\label{exa:SeparableBPCounterexample}
We construct a concrete separable, nuclear counterexample to the second and third Brown-Pedersen conjectures.
Let $B$ be the stable Kirchberg algebra in the UCT class and with $K_0(B) = 0$ and $K_1(B) \cong \ZZ$.
(This algebra is unique by the Kirchberg-Phillips classification theorem, \cite[Theorem~8.4.1]{Ror02Classification}, and it exists by \cite[Proposition~4.3.3]{Ror02Classification}.)

Next, we compute the extension group $\Ext(B,\Cpct)$.
Applying \cite[Proposition~17.6.5]{Bla98KThy} at the first step, and using the Universal Coefficient Theorem (\cite[Theorem~23.1.1]{Bla98KThy}) at the second step, we have
\[
\Ext(B,\Cpct)
\cong KK^1(B,\Cpct)
\cong \Hom(K_\ast(B),K_\ast(\Cpct))
\cong \Hom(\ZZ,\ZZ)
\cong \ZZ.
\]

Using Voiculescu's theorem on the existence of absorbing extensions (\cite[Theorem~15.12.3]{Bla98KThy}), we realize the element $1 \in \ZZ \cong \Ext(B,\Cpct)$ by an essential, nonunital extension
\[
0 \to \Cpct \to A \to B \to 0.
\]

We show that $A$ is a separable, nuclear \ca{} with real rank zero, trivial $K$-theory, and such that
\[
\rr(M(A)) = \rr(M(A)/A) = 1.
\]

First, $A$ is an extension of separable, nuclear \ca{s} and therefore separable and nuclear itself;
see \cite[Proposition~IV.3.1.3]{Bla06OpAlgs}.
Further, since $\Cpct$ and $B$ have real rank zero, and since the induced map $K_0(A) \to K_0(B)$ is surjective, it follows that $A$ has real rank zero;
see \cite[Proposition~4]{LinRor95ExtLimitCircle}.

Since the extension realizes the class $1 \in \ZZ \cong \Ext(B,\Cpct)$, the associated index map $\ZZ \cong K_1(B) \to K_0(\Cpct) \cong \ZZ$ is an isomorphism.
Using the six-term exact sequence in $K$-theory (\cite[Corollary~V.1.2.22]{Bla06OpAlgs}), we get $K_0(A) = K_1(A) = 0$.

Using \autoref{prp:rrMultExtByKK} at the first step, and applying \autoref{prp:rrMultStableSimple} at the second step (using that $K_1(B)\neq 0$), we have
\[
\rr(M(A)) = \rr(M(B)) = 1.
\]

To see that $\rr(M(A)/A)=1$, we consider the extension
\[
0 \to A \to M(A) \to M(A)/A \to 0.
\]

Since the real rank does not increase when passing to quotients, we deduce that $\rr(M(A)/A) \leq 1$;
see \autoref{prp:EstimateWithXRR}.
To reach a contradiction, assume that $\rr(M(A)/A)=0$.
Using the six-term exact sequence in $K$-theory, and using that $K_1(A)=0$, we see that the map $K_0(M(A)) \to K_0(M(A)/A))$ is surjective, and then $\rr(M(A))=0$ by \cite[Proposition~4]{LinRor95ExtLimitCircle}, which is the desired contradiction.
Thus, $M(A)/A$ does not have real rank zero, and so $\rr(M(A)/A)=1$.
\end{exa}

%==========================================================================================
\begin{rmk}
\label{rmk:LargerBPCounterexample}
In \cite{Osa93CounterexBrownPedConj}, Osaka found non-$\sigma$-unital counterexamples to the second and third Brown-Pedersen conjectures. 
Indeed, while for a $\sigma$-compact, locally compact, Hausdorff space $X$, we have $\locdim(X)=\dim(\beta X)$, there exists a non-$\sigma$-compact, locally compact, Hausdorff space $Y$ with $\locdim(Y)=0$ and $\dim(\beta Y) \geq 1$;
see \cite{Osa93CounterexBrownPedConj} and \cite[Section~3.3]{DowHar22ZeroDimFSpNotStrZeroDim}.
Then the commutative \ca{} $C_0(Y)$ satisfies
\[
\rr(C_0(Y)) = 0, \quad
K_1(C_0(Y)) = 0, \andSep
\rr(M(C_0(Y))) \geq 1.
\]

It is a general phenomenon that multiplier algebras of non-$\sigma$-unital \ca{s} can behave rather strange.
For example, there exists a non-$\sigma$-unital \ca{} $A$ such that $M(A)$ agrees with the minimal unitization of $A$; 
see \cite{GhaKos18ExtCpctOpsTrivMult}.
In this case, the corona algebra $M(A)/A$ is isomorphic to $\mathbb{C}$, and $A$ is complemented in $M(A)$ as a Banach space. 

On the other hand, if $B$ is a nonunital, $\sigma$-unital \ca{}, then $B$ is not complemented in $M(B)$ as a Banach space (\cite[Corollary~3.7]{Tay72GenPhillipsThm}), and the corona algebra $M(B)/B$ is not separable (\cite[Corollary~2, Theorem~13]{Ped86SAW}).

Therefore, $\sigma$-unitality is a common and natural assumption when working with multiplier algebras.
\end{rmk}

%==========================================================================================
As another major application of \autoref{prp:rrMultExtByKK}, we compute the real rank of the stable multiplier and stable corona algebra of~$\Bdd$.
Since $\Bdd\otimes\Cpct$ has real rank zero and trivial $K_1$-group, this provides another natural counterexample to the second and third Brown-Pedersen conjectures.

%==========================================================================================
\begin{thm}
\label{prp:rrStableMultBdd}
We have 
\[
\rr( M(\Bdd\otimes\Cpct) )
= \rr( M(\Bdd\otimes\Cpct)/(\Bdd\otimes\Cpct) )
= 1.
\] 

Further, we have $\xrr(\Bdd\otimes\Cpct)=0$.
Thus, given any extension
\[
0 \to \Bdd\otimes\Cpct \to E \to B \to 0,
\]
we have $\rr(E) = \rr(B)$.
\end{thm}
\begin{proof}
By tensoring the extension $0 \to \Cpct \to \Bdd \to \Calk \to 0$ by $\Cpct$, and identifying $\Cpct\otimes\Cpct$ with $\Cpct$, we obtain the extension
\[
0 \to \Cpct \to \Bdd\otimes\Cpct \to Q\otimes\Cpct \to 0.
\]
Using \autoref{prp:rrMultExtByKK} at the first step, and \autoref{exa:StableMultCalkin} at the second step, we get
\[
\rr( M(\Bdd\otimes\Cpct) )
= \rr( M(Q\otimes\Cpct) ) 
= 1.
\]

By \cite[Proposition~3.11]{Thi23arX:RRExt}, the extension real rank is dominated by the real rank of the multiplier algebra.
Thus, we have $\xrr(\Bdd\otimes\Cpct) \leq 1$.
We also have $\rr(\Bdd\otimes\Cpct)=0$ and $K_1(\Bdd\otimes\Cpct)=0$, and so $\xrr(\Bdd\otimes\Cpct)=0$ by \autoref{prp:xrr0SufficientCond}.
The statement about the real rank of extensions by $\Bdd\otimes\Cpct$ follows from \autoref{prp:EstimateWithXRR}.

Finally, to compute the real rank of the stable corona algebra of $\Bdd$, we consider the extension
\[
0 \to \Bdd\otimes\Cpct \to M(\Bdd\otimes\Cpct) \to M(\Bdd\otimes\Cpct)/(\Bdd\otimes\Cpct) \to 0.
\]
Using at the first step that $\xrr(\Bdd\otimes\Cpct)=0$ (and \autoref{prp:EstimateWithXRR}), we get
\[
\rr( M(\Bdd\otimes\Cpct)/(\Bdd\otimes\Cpct) )
= \rr( M(\Bdd\otimes\Cpct) ) 
= 1. \qedhere
\]
\end{proof}

%==========================================================================================
\begin{rmk}
In the setting of \autoref{prp:rrMultExtByKK}, if $\rr(M(B)/B)=0$, then the real rank of $M(E)/E$ takes the value $0$ or $1$, and both are possible.
The value $0$ arises for example from trivial extensions.
On the other hand, for the extension 
\[
0 \to \Cpct\to \Bdd\otimes\Cpct \to \Calk\otimes\Cpct \to 0
\]
we have
\[
\rr( M(\Bdd\otimes\Cpct)/(\Bdd\otimes\Cpct) ) = 1, \andSep
\rr( M(\Calk\otimes\Cpct)/(\Calk\otimes\Cpct) ) = 0
\]
by \autoref{prp:rrStableMultBdd} and \autoref{exa:StableMultCalkin}.
\end{rmk}

%==========================================================================================
\begin{exa}
\label{exa:tensMax}
%As before, let $\Bdd$ denote the algebra of bounded operators on a separable, infinite-dimensional Hilbert space, and let $Q:=\Bdd/\Cpct$ denote the Calkin algebra.
We have
\[
\rr( \Bdd\tensMax \Bdd )
= \rr( \Bdd\tensMax \Calk )
= \max\big\{ 1, \rr( \Calk \tensMax \Calk ) \big\}.
\]
Indeed, since the maximal tensor product preserves short exact sequences (\cite[II.9.6.6]{Bla06OpAlgs}), we have an extension
\[
0 \to \Bdd\tensMax\Cpct \to \Bdd\tensMax\Bdd \to \Bdd\tensMax\Calk \to 0,
\]
to which we may apply \autoref{prp:rrStableMultBdd} to obtain that $\Bdd\tensMax\Bdd$ and $\Bdd\tensMax\Calk$ have the same real rank.

Next, we consider the extension
\[
0 \to \Cpct\tensMax Q \to \Bdd\tensMax Q \to Q\tensMax Q \to 0.
\]
Using that $\xrr( \Calk\otimes\Cpct) = 1$ by \autoref{exa:StableMultCalkin}, and applying \autoref{prp:EstimateWithXRR}, we have
\[
\rr( \Calk\tensMax\Calk ) 
\leq \rr( \Bdd\tensMax\Calk ) 
%\leq \max\big\{ \xrr(\Calk\otimes\Cpct), \rr( \Calk\tensMax\Calk ) \big\}
\leq \max\big\{ 1, \rr( \Calk\tensMax\Calk ) \big\}.
\]
On the other hand, applying \cite[Corollary~1.2]{Osa99NonzeroRR} at the first step, and using at the second step that the minimal tensor product $\Bdd\otimes\Bdd$ is a quotient of $\Bdd\tensMax\Bdd$, we have
\[
1
\leq \rr( \Bdd\otimes\Bdd )
\leq \rr( \Bdd\tensMax\Bdd )
= \rr( \Bdd\tensMax\Calk ).
\]
We deduce that the real rank of $\Bdd\tensMax\Calk$ is equal to $\max\{ 1, \rr( \Calk\tensMax\Calk ) \}$.
\end{exa}

%==========================================================================================
\begin{qst}
What is the real rank of $Q\tensMax Q$?
\end{qst}

%==========================================================================================
If $\rr(Q\tensMax Q)=0$, then \autoref{exa:tensMax} would imply that $\rr(\Bdd\tensMax \Bdd)=\rr(\Bdd\otimes\Bdd)=1$, which would answer \cite[Question~3.3]{Osa99NonzeroRR}.
I suspect, however, that the real rank of $Q\tensMax Q$ is nonzero.
On the other hand, note that the minimal tensor product $Q\otimes Q$ is simple, unital and purely infinite, whence $\rr(Q\otimes Q)=0$ by Zhang's theorem \cite[Proposition~V.3.2.12]{Bla06OpAlgs}.

%==========================================================================================
%==========================================================================================
\section{Extensions by simple, purely infinite C*-algebras}
\label{sec:ExtSimplePI}

%==========================================================================================
In this section, we consider extensions
\[
0 \to A \to E \to B \to 0
\]
where $E$ is $\sigma$-unital, and $A$ is simple and purely infinite.
We solve \autoref{pbm:xrrHerMA} for~$A$ (\autoref{prp:xrrHerMultPI}) and deduce in \autoref{prp:rrMultExtSimplePI} that
\[
\rr(M(B)) \leq \rr(M(E)) \leq \max\big\{ 1,\rr(M(B) \big\}.
%, \andSep \rr(M(E)/E) \leq \max\{1,\rr(M(B)/B),
\]

Thus, we have $\rr(M(E)) = \rr(M(B))$, unless $\rr(M(B))=0$ while $\rr(M(E))=1$, and we will see in \autoref{exa:rrMultExtSimplePI-ExceptionalCase} that this exceptional case can occur.
We also show that $\rr(M(E)) = \rr(M(B))$ if we additionally assume that $K_1(A)=0$.
We obtain analogous results for the real rank of the corona algebra $M(E)/E$.

\medskip

Since we do not want to assume that $A$ is separable or $\sigma$-unital, we first devise a method to reduce the problem of computing $\rr(M(E))$ to suitable subextensions $0 \to A' \to E' \to B' \to 0$, where $A'$ is separable.
To that end, we show that for a $\sigma$-unital \ca{} $E$, the multiplier algebra $M(E)$ is exhausted by the images of maps $M(D) \to M(E)$ for separable sub-\ca{s} $D \subseteq E$ containing an approximate unit of $E$.

Here we use that a morphism $\varphi \colon D \to E$ between \ca{s} naturally induces a unital morphism $\bar{\varphi} \colon M(D) \to M(E)$ if the image of $\varphi$ contains an approximate unit for $E$.

%==========================================================================================
\begin{lma}
\label{prp:ReduceToSeparable}
Let $E$ be a $\sigma$-unital \ca{}, and let $L \subseteq M(E)$ be a separable sub-\ca{}.
Then there exists a separable sub-\ca{} $D \subseteq E$ containing an approximate unit for $E$ and such that the image of the naturally induced map $M(D) \to M(E)$ contains $L$.
\end{lma}
\begin{proof}
Let $(e_n)_n$ be a countable approximate unit for $E$, and let $L_0 \subseteq L$ be a countable dense subset.
Let $D$ be the sub-\ca{} of $E$ generated by 
\[
\big\{ e_n : n \in\NN \big\} 
\cup \big\{ e_n a : n \in \NN, a \in L_0 \big\} 
\cup \big\{ a e_n : n \in \NN, a \in L_0 \big\}.
\]
Then $D$ is a separable sub-\ca{} of $E$ containing the approximate unit $(e_n)_n$.
The inclusion $\iota \colon D \to E$ extends to an injective homomorphism $\iota^{**} \colon D^{**} \to E^{**}$.
We realize the multiplier algebra of $D$ as 
\[
M(D) = \big\{ x \in D^{**} : xD + Dx \subseteq D \big\},
\]
and similarly for $M(E)$.
Since $D$ contains an approximate unit for $E$, the map $\iota^{**}$ sends $M(D)$ into $M(E)$.
We may therefore view $M(D)$ as a subalgebra of $M(E)$, as shown in the following diagram:
\[
\xymatrix{ %@R-10pt{
E \ar@{}[r]|{\subseteq} 
& M(E) \ar@{}[r]|{\subseteq} 
& E^{**} \\
D \ar@{}[r]|{\subseteq} \ar@{}[u]|{\subseteqRotatedUp}
& M(D) \ar@{}[r]|{\subseteq} \ar@{}[u]|{\subseteqRotatedUp}
& D^{**}. \ar@{}[u]|{\subseteqRotatedUp}
}
\]

To verify $L \subseteq M(D)$, it suffices to show that $M(D)$ contains $L_0$.
Let $a \in L_0$, $d \in D$, and $\varepsilon>0$.
Since $(e_n)_n$ is an approximate unit for $E$ (and hence for~$D$), we obtain $n \in \NN$ such that $\| e_nd  - d \| < \varepsilon/\|a\|$. 
Then
\[
\| ad - ae_n d \| < \varepsilon.
\]
Since $ae_n$ belongs to $D$ by construction, we have $ae_nd \in D$.
Thus, $ad$ has distance less than $\varepsilon$ to $D$.
Since this holds for every $\varepsilon>0$, we get $ad \in D$, and thus $aD \subseteq D$.
Similarly, we obtain $Da \subseteq D$, and thus $a \in M(D)$.
\end{proof}

%==========================================================================================
\begin{pgr}
\label{pgr:LS}
Let $A$ be a (nonseparable) \ca{}.
We use $\Sep(A)$ to denote the collection of separable sub-\ca{s} of $A$, equipped with the partial order given by inclusion.
Following the terminology from model theory, we say that a family $\mathcal{F} \subseteq \Sep(A)$ is a \emph{club} if it is $\sigma$-complete (for every countable, upward directed subset $\mathcal{F}_0 \subseteq \mathcal{F}$, we have $\overline{\bigcup\mathcal{F}_0} \in \mathcal{F}$) and cofinal (for every $B \in \Sep(A)$ there exists $C \in \mathcal{F}$ with $B \subseteq C$);
see \cite[Section~6.2]{Far19BookSetThyCAlg}.

A property $\mathcal{P}$ for \ca{s} is said to satisfy the L{\"o}wenheim-Skolem condition if for every \ca{} $A$ satisfying $\mathcal{P}$ there is a club in $\Sep(A)$ of (separable) \ca{s} satisfying $\mathcal{P}$.
Many common properties of \ca{s} satisfy the L{\"o}wenheim-Skolem condition, including properties that are axiomatizable in model theory \cite[Section~7.3]{Far19BookSetThyCAlg}, and properties that are `separably inheritable' in the sense of Blackadar \cite[Definition~II.8.5.1]{Bla06OpAlgs};
see, for example, \cite[Paragraph~4.5]{Thi23arX:RRExt}.

A countable intersection of clubs in $\Sep(A)$ is again a club, which shows that the conjunction of countably many properties with the L{\"o}wenheim-Skolem condition also satisfies the L{\"o}wenheim-Skolem condition.
\end{pgr}

%==========================================================================================
\begin{lma}
\label{prp:rrMultFromSepSub}
Let $A$ be a $\sigma$-unital \ca, and let $n \in \NN$.
Then:

(1)
Assume that for every separable sub-\ca{} $B \subseteq A$ there exists a sub-\ca{} $D \subseteq A$ with $B \subseteq D$ and $\rr(M(D)) \leq n$.
Then $\rr(M(A)) \leq n$.

(2)
Assume that for every separable sub-\ca{} $B \subseteq A$ there exists a sub-\ca{} $D \subseteq A$ such that $B \subseteq D$ and $\rr(M(D)/D) \leq n$.
Then $\rr(M(A)/A) \leq n$.
\end{lma}
\begin{proof}
We only verify~(1).
The proof of~(2) is analogous.
%Let $\mathcal{F} \subseteq \Sep(A)$ be a club such that $\rr(M(B)) \leq n$ for each $B \in \mathcal{F}$.
Let $(a_0,\ldots,a_n) \in M(A)_\sa^{k+1}$ be a self-adjoint $(n+1)$-tuple in~$M(A)$, and let $\varepsilon>0$.
We need to find a unimodular tuple $(b_0.\ldots,b_n) \in M(A)_\sa$ such that
\[
\| b_0 - a_0 \| < \varepsilon, \quad \ldots,\quad 
\| b_n - a_n \| < \varepsilon.
\]

By \autoref{prp:ReduceToSeparable}, there exists a separable sub-\ca{} $B \subseteq A$ containing an approximate unit for $A$ such that the image of the induced inclusion $M(B) \to M(A)$ contains $a_0,\ldots,a_k$.
By assumption, we obtain a sub-\ca{} $D \subseteq A$ such that $B \subseteq D$ and $\rr(M(D)) \leq n$.
Then $D$ contains an approximate unit for $A$, and the image of the inclusion $M(D) \to M(A)$ also contains $a_0,\ldots,a_k$.
Using that $\rr(M(D)) \leq n$, we can find the desired tuple $(b_0.\ldots,b_n)$ in $M(D)$.
\end{proof}

%==========================================================================================
\begin{exa}
%\label{prp:SeparableBPCounterexample}
We can apply \autoref{prp:rrMultFromSepSub} to show that $\Bdd\otimes\Cpct$ contains many \emph{separable} sub-\ca{s} that are counterexamples to the second and third Brown-Pedersen conjectures.
(See \autoref{exa:SeparableBPCounterexample} for a concrete separable and nuclear counterexample.)

By \autoref{prp:rrStableMultBdd}, we have 
\[
\rr(\Bdd\otimes\Cpct)=0, \quad 
K_1(\Bdd\otimes\Cpct)=0, \andSep 
\rr(M(\Bdd\otimes\Cpct)/\Bdd\otimes\Cpct) = 1.
\]

Since `real rank zero' and `vanishing $K_0$-group' each satisfy the L{\"o}wenheim-Skolem condition (\cite[Paragraph~II.8.5.5]{Bla06OpAlgs}), there exists a club $\mathcal{F}$ of separable sub-\ca{s} $A \subseteq \Bdd\otimes\Cpct$ satisfying $\rr(A)=0$ and $K_1(A)=0$.
If every $A \in \mathcal{F}$ satisfied $\rr(M(A)/A)=0$, then \autoref{prp:rrMultFromSepSub} would imply $\rr(M(\Bdd\otimes\Cpct)/\Bdd\otimes\Cpct) = 0$, a contradiction.
Thus, there exists $A \in \mathcal{F}$ such that
\[
\rr(A)=0, \quad 
K_1(A)=0, \andSep 
\rr(M(A)/A) \neq 0.
\]
\end{exa}

%==========================================================================================
We now turn to a technical result that allows us to estimate the real rank of the multiplier and corona algebra of an extension with a nonseparable ideal.

%==========================================================================================
\begin{lma}
\label{prp:rrMultUsingClub}
Let $0 \to A \to E \xrightarrow{\pi} B \to 0$ be an extension of \ca{s}, let $n \in \NN$, and assume that $E$ is $\sigma$-unital.
Then:

(1)
If $A$ contains a club of separable sub-\ca{s} $A' \subseteq A$ such that every hereditary sub-\ca{} $J$ of $M(A')$ satisfies $\xrr(J) \leq n$, then
\[
\rr(M(B))
\leq \rr(M(E)) 
\leq \max\big\{ n, \rr(M(B)) \big\}.
\]

(2)
If $A$ contains a club of separable sub-\ca{s} $A' \subseteq A$ such that every hereditary sub-\ca{} $K$ of $M(A')/A'$ satisfies $\xrr(K) \leq n$, then
\[
\rr(M(B)/B)
\leq \rr(M(E)/E) 
\leq \max\big\{ n, \rr(M(B)/B) \big\}.
\]
\end{lma}
\begin{proof}
We identify $A$ with an ideal in $E$, and we let $\bar{\pi}\colon M(E) \to M(B)$ be the natural morphism induced by $\pi$.
Since $E$ is $\sigma$-unital, $\bar{\pi}$ is surjective by \cite[Theorem~10]{Ped86SAW}.
It follows that the natural map $M(E)/E \to M(B)/B$ is surjective as well. 
By \autoref{prp:EstimateBasic}, we have
\[
\rr(M(B)) \leq \rr(M(E)), \andSep
\rr(M(B)/B) \leq \rr(M(E)/E).
\]

We now verify the upper bounds.
Set $k := \max\{ n, \rr(M(B)) \}$, which we may assume to be finite.
Let $(a_0,\ldots,a_k) \in M(E)_\sa^{k+1}$ be a self-adjoint $(k+1)$-tuple in~$M(E)$, and let $\varepsilon>0$.
We need to find a unimodular tuple $(b_0.\ldots,b_n) \in M(E)_\sa$ such that
\[
\| b_0 - a_0 \| < \varepsilon, \quad \ldots,\quad 
\| b_k - a_k \| < \varepsilon.
\]

Since $\rr(M(B)) \leq k$, the tuple $(\bar{\pi}(a_0),\ldots,\bar{\pi}(a_k)) \in M(B)^{k+1}_\sa$ can be approximated by a unimodular, self-adjoint tuple, which then may be lifted to a self-adjoint tuple in $M(E)$ close to $(a_0,\ldots,a_k)$.
Thus, without loss of generality, we may assume that $(\bar{\pi}(a_0),\ldots,\bar{\pi}(a_k))$ is unimodular. %belongs to $\Lg_{k+1}(M(B))_\sa$.

\medskip

By \autoref{prp:ReduceToSeparable}, there exists a separable sub-\ca{} $D \subseteq A$ containing an approximate unit for $E$ such that the image of the induced inclusion $M(D) \to M(E)$ contains $a_0,\ldots,a_k$.
Then $D \cap A$ is a separable sub-\ca{} of $A$.
By \cite[Lemma~3.2(1)]{Thi23grSubhom}, if $S \subseteq \Sep(A)$ is a club of separable sub-\ca{s}, then so is $\{ E' \in \Sep(E) : E' \cap A \in S \}$.
Using the assumption, we may therefore assume that $D \cap A$ has the property that every hereditary sub-\ca{} $J$ of $M(D \cap A)$ satisfies $\xrr(J) \leq n$.
Let $\varrho$ denote the restriction of $\pi$ to $D$.
We have the following inclusion of extensions:
\[
\xymatrix@R-10pt{
0 \ar[r]
& A \ar[r]
& E \ar[r]^{\pi}
& B \ar[r]
& 0 \\
0 \ar[r] 
& D \cap A \ar[r] \ar@{}[u]|{\subseteqRotatedUp}
& D \ar[r]^-{\varrho} \ar@{}[u]|{\subseteqRotatedUp}
& D/(D \cap A) \ar[r] \ar@{}[u]|{\subseteqRotatedUp}
& 0.
}
\]

Since $D$ is separable (and hence $\sigma$-unital), the surjective morphism $\varrho$ naturally induces a surjective morphism $\bar{\varrho} \colon M(D) \to M(D/(D \cap A))$, and as explained in \autoref{pgr:MultExt} and \autoref{prp:KernelMultExt} the kernel $J$ of~$\bar{\varrho}$ is isomorphic to a hereditary sub-\ca{} of $M(D \cap A)$.
We have the following commutative diagram:
\[
\xymatrix@R-10pt{
& & M(E) \ar[r]^{\bar{\pi}}
& M(B) \\
0 \ar[r] 
& J \ar[r] 
& M(D) \ar[r]^-{\bar{\varrho}} \ar@{}[u]|{\subseteqRotatedUp}
& M(D/(D \cap A)) \ar[r] \ar@{}[u]|{\subseteqRotatedUp}
& 0.
}
\]

By construction, we have $\xrr(J) \leq n \leq k$.
Further, the tuple $(a_0,\ldots,a_k)$ belongs to $M(D)^{k+1}_\sa$.
Let us see that $(\bar{\varrho}(a_0),\ldots,\bar{\varrho}(a_k))$ is unimodular.
Since $(\bar{\pi}(a_0),\ldots,\bar{\pi}(a_k))$ is unimodular, there exists $\delta > 0$ such that
\[
\sum_{j=0}^k \bar{\pi}(a_j)^2 \geq \delta 1_{M(B)}.
\]
It follows that
\[
\sum_{j=0}^k \bar{\varrho}(a_j)^2 \geq \delta 1_{M(D/(D \cap A))},
\]
and thus $(\bar{\varrho}(a_0),\ldots,\bar{\varrho}(a_k))$ is unimodular.
Now, since $\xrr(J) \leq k$, it follows from \autoref{dfn:xrr} that $(a_0,\ldots,a_k)$ can be approximated by an unimofular, self-adjoint tuple in $M(D)$.
The image of this tuple in $M(E)$ has the desired properties.
This verifies~(1).
The proof of~(2) is similar.
We omit the details.
\end{proof}

%==========================================================================================
\begin{lma}
\label{prp:xrrHerMultPI}
Let $A$ be a $\sigma$-unital, simple, purely infinite \ca{}.
Then:
\begin{enumerate}
\item
If $J \subseteq M(A)$ is a hereditary sub-\ca{}, then $\xrr(J) \leq 1$.
If we additionally assume that $K_1(A)=0$, then $\xrr(J)=0$.
\item
If $K \subseteq M(A)/A$ is a hereditary sub-\ca{}, then $\xrr(K) \leq 1$.
If we additionally assume that $K_0(A)=0$, then $\xrr(K)=0$.
\end{enumerate}
\end{lma}
\begin{proof}
By \cite[Theorem~1.2]{Zha92RR0CoronaMultiplier1}, $A$ is either unital or stable.
In the first case, we have $M(A)=A$ and $M(A)/A=\{0\}$, and then~(2) clearly holds.
To verify~(1), let $J \subseteq M(A)=A$ be a hereditary sub-\ca{}.
The result is clear if $J=\{0\}$.
If $J$ is nonzero, then it is Morita equivalent to $A$ and therefore simple and purely infinite.
Then $\xrr(J) \leq 1$ by \cite[Proposition~5.11(2)]{Thi23arX:RRExt}.
If additionally $K_1(A)=0$, then also $K_1(J)=0$, and then $\xrr(J)=0$ by \cite[Proposition~5.9(2)]{Thi23arX:RRExt}.

We may therefore assume that $A$ is stable.
Then it follows from results of Zhang that $M(A)/A$ is simple and purely infinite.
Indeed, $M(A)/A$ is simple by \cite[Theorem~3.3]{Zha90RieszDecomp}.
Further, every hereditary sub-\ca{} of $M(A)/A$ contains an infinite projection by \cite[Theorem~1.3(a)]{Zha89StrProjIdlsCorona}, and it is known that this characterizes pure infiniteness for simple \ca{s};
see, fore example \cite[Proposition~V.2.3.3]{Bla06OpAlgs}.
See also \cite{Lin04SimpleCorona}.

To verify~(2), let $K \subseteq M(A)/A$ be a hereditary sub-\ca{}.
Without loss of generality, we may assume that $K$ is nonzero.
Then $K$ is Morita equivalent to $M(A)/A$, whence $K$ is simple and purely infinite, and then $\xrr(K) \leq 1$ by \cite[Proposition~5.11(2)]{Thi23arX:RRExt}.

Let us now additionally assume that $K_0(A)=0$.
Since $A$ is stable, we have $K_0(M(A))=K_1(M(A))=0$; 
see, for example, \cite[Theorem~10.2]{Weg93KThyBook}.
Using the six-term exact sequence in $K$-theory (\cite[Corollary~V.1.2.22]{Bla06OpAlgs}) at the second step, we get
\[
K_1(K) \cong K_1(M(A)/A) \cong K_0(A) = 0,
\]
and then $\xrr(K) = 0$ by \cite[Proposition~5.9(2)]{Thi23arX:RRExt}.

To verify~(1), let $J \subseteq M(A)$ be a hereditary sub-\ca{}.
We have the following inclusion of extensions:
\[
\xymatrix@R-15pt{
0 \ar[r] 
& A \ar[r] 
& M(A) \ar[r] 
& M(A)/A \ar[r] 
& 0 \\
0 \ar[r] 
& J \cap A \ar[r] \ar@{}[u]|{\subseteqRotatedUp}
& J \ar[r] \ar@{}[u]|{\subseteqRotatedUp}
& J/(J \cap A) \ar[r] \ar@{}[u]|{\subseteqRotatedUp}
& 0 \\
}
\]

Since $J\cap A \subseteq A$ and $J/( J\cap A) \subseteq M(A)/A$ are hereditary sub-\ca{s}, they are simple and purely infinite (or the zero algebra), and hence have extension real rank at most one by \cite[Proposition~5.11(2)]{Thi23arX:RRExt},
Using \autoref{prp:xrrOfExtension} at the first step, we get
\[
\xrr(J) 
\leq \max\big\{ \xrr(J \cap A), \xrr(J/(J \cap A)) \big\}
\leq 1.
\]

Let us now additionally assume that $K_1(A)=0$.
Then $M(A)$ has real rank zero by \cite[Corollary~2.6(ii)]{Zha92RR0CoronaMultiplier1}.
Since real rank zero passes to hereditary sub-\ca{s} by \cite[Corollary~2.8]{BroPed91CAlgRR0}, we get $\rr(J)=0$.
Further, we have $K_1(J)=0$ by \cite[Theorem~2.4]{Zha90TrivialK1Flow}.
Then $\xrr(J)=0$ by \autoref{prp:xrr0SufficientCond}.
\end{proof}

%==========================================================================================
We stress that in the next result, we do not assume that $A$ is separable or $\sigma$-unital.

%==========================================================================================
\begin{prp}
\label{prp:rrMultExtSimplePI}
Let $0 \to A \to E \to B \to 0$ be an extension of \ca{s}.
Assume that $E$ is $\sigma$-unital, and that $A$ is simple and purely infinite.
Then:
\begin{enumerate}
\item
We have
\[
\rr(M(B)) \leq \rr(M(E)) \leq \max\big\{1, \rr(M(B)) \big\}.
\]
If we additionally assume that $K_1(A)=0$, then
\[
\rr(M(E)) = \rr(M(B)),
\]
\item
We have
\[
\rr(M(B)/B) \leq \rr(M(E)/E) \leq \max\big\{1, \rr(M(B)/B) \big\}.
\]
If we additionally assume that $K_0(A)=0$, then 
\[
\rr(M(E)/E) = \rr(M(B)/B).
\]
\end{enumerate}
\end{prp}
\begin{proof}
Since the property `simple and purely infinite' satisfies the L{\"o}wenheim-Skolem condition (\cite[Example~7.3.4]{Far19BookSetThyCAlg}), $A$ contains a club $\mathcal{F}$ of separable sub-\ca{s} of $A$ that are simple and purely infinite.
For each $A' \in \mathcal{F}$ and for all hereditary sub-\ca{s} $J \subseteq M(A')$ and $K \subseteq M(A')/A'$, we have $\xrr(A) \leq 1$ and $\xrr(K) \leq 1$ by \autoref{prp:xrrHerMultPI}.
Now the estimates for $\rr(M(E))$ and $\rr(M(E)/E)$ follow from \autoref{prp:rrMultUsingClub}.

Let us now additionally assume that $K_1(A)=0$.
Since the property `vanishing $K_1$-group' satisfies the L{\"o}wenheim-Skolem condition (\cite[Paragraph~II.8.5.5]{Bla06OpAlgs}), we obtain a club $\mathcal{F}_1$ of separable sub-\ca{s} of $A$ that are simple, purely infinite and have vanishing $K_1$-group.
For each $A' \in \mathcal{F}_1$ and every hereditary sub-\ca{} $J \subseteq M(A')$, we have $\xrr(A) \leq 0$ by \autoref{prp:xrrHerMultPI}.
Then $\rr(M(E))=\rr(M(B))$ by \autoref{prp:rrMultUsingClub}.

If instead we additionally assume that $K_0(A)=0$, then a similar argument (using that vanishing $K_0$-group satisfies the L{\"o}wenheim-Skolem condition) gives $\rr(M(E)/E) = \rr(M(B)/B)$.
\end{proof}

%==========================================================================================
\begin{exa}
\label{exa:rrMultExtSimplePI-ExceptionalCase}
Let $\Calk$ denote the Calkin algebra, and consider the extension
\[
0 \to \Calk\otimes\Cpct \to M(\Calk\otimes\Cpct) \to M(\Calk\otimes\Cpct)/(\Calk\otimes\Cpct) \to 0.
\]
Then $\Calk\otimes\Cpct$ is simple and purely infinite.
Further, $M(\Calk\otimes\Cpct)$ and $M(\Calk\otimes\Cpct)/(\Calk\otimes\Cpct)$ are unital and therefore agree with their multiplier algebras.
By \autoref{exa:StableMultCalkin}, we have 
\[
\rr( M(\Calk\otimes\Cpct) ) = 1, \andSep
\rr( M(\Calk\otimes\Cpct)/(\Calk\otimes\Cpct) ) = 0.
\]

This shows that in \autoref{prp:rrMultExtSimplePI}(1), the exceptional case $\rr(M(B))=0$ and $\rr(M(E))=1$ can occur.
\end{exa}

%==========================================================================================
%==========================================================================================
\section{Extensions by certain simple C*-algebras with stable rank one}
\label{sec:ExtSimpleSR1}

%==========================================================================================
In this section, we consider extensions
\[
0 \to A \to E \to B \to 0
\]
where $E$ is $\sigma$-unital, and $A$ is simple, with real rank zero, stable rank one, strict comparison of positive elements by traces, and finitely many extremal traces (normalized at some nonzero projection).
We solve \autoref{pbm:xrrHerMA} for~$A$ (\autoref{prp:xrrHerMultTraces}) and deduce in \autoref{prp:rrMultExtSimpleSR1} that
\[
\rr(M(B)) \leq \rr(M(E)) \leq \max\big\{ 1,\rr(M(B) \big\}.
%, \andSep \rr(M(E)/E) \leq \max\{1,\rr(M(B)/B),
\]

Using this, we compute in \autoref{prp:rrStableMultIIinf} the real rank of the stable multiplier and stable corona algebra of a countably decomposable type~\IIinf{} factor $N$ as
\[
\rr( M(N\otimes\Cpct) )
= \rr( M(N\otimes\Cpct)/(N\otimes\Cpct) )
= 1.
\]
Together with \autoref{prp:rrStableMultBdd} and results of Zhang \cite{Zha92RR0CoronaMultiplier1} this completes the computation of the real rank
for stable multiplier and corona algebras of countably decomposable factors;
see \autoref{prp:rrMultFactor}

\medskip

%==========================================================================================
Following Kirchberg-R{\o}rdam \cite[Definition~4.1]{KirRor00PureInf}, we say that a \ca{}~$A$ is \emph{purely infinite} if $A$ has no characters and if for any $a,b \in A_+$ such that $a$ belongs to the ideal generated by $b$, then $a$ is Cuntz subequivalent to $b$.
It is known that \emph{simple}, purely infinite \ca{s} have real rank zero;
see \cite[Proposition~V.3.2.12]{Bla06OpAlgs}.
It is also known that non-simple, purely infinite \ca{s} need not have real rank zero;
see, for example, \cite{PasRor07PIRR0}.
The next result shows in particular, that purely infinite \ca{s} with finite primitive ideal space have real rank at most one.

%==========================================================================================
\begin{lma}
\label{prp:xrrHerPIFinIdl}
Purely infinite \ca{s} with finite primitive ideal space have extension real rank $\leq 1$.
%Then $\xrr(J) \leq 1$ for every hereditary sub-\ca{} $J \subseteq A$.
\end{lma}
\begin{proof}
By induction over $n$, we show that the result holds for every purely infinite \ca{} whose primitive ideal space contains at most $n$ elements.
To show that case $n=1$, let $A$ be a purely infinite \ca{} with at most one primitive ideal.
Then $A$ is simple (and purely infinite), and therefore $\xrr(A) \leq 1$ by \cite[Proposition~5.11(2)]{Thi23arX:RRExt}.

Next, assume that the result holds for some $n$, and let $A$ be a purely infinite \ca{} whose primitve ideal space contains at most $n+1$ elements.
Let $I \subseteq A$ be any ideal with $I\neq\{0\}$ and $I \neq A$.
By \cite[Theorem~4.19]{KirRor00PureInf}, both $I$ and $A/I$ are purely infinite, and their primitive ideal spaces have at most $n$ elements.
By assumption of the induction, we get $\xrr(I) \leq 1$ and $\xrr(A) \leq 1$, and then $\xrr(A) \leq 1$ by \autoref{prp:xrrOfExtension}.
\end{proof}

%==========================================================================================
A \ca{} $A$ is said to have \emph{strict comparison of positive elements by traces} if for any $a,b \in (A\otimes\Cpct)_+$ such that $d_\tau(a) \leq (1-\varepsilon)d_\tau(b)$ for some $\varepsilon>0$ and all lower semicontinuous, $[0,\infty]$-valued traces $\tau$ on $A$, then $a$ is Cuntz subequivalent to $b$.
By \cite[Remark~3.7]{NgRob16CommutatorsPureCa} and \cite[Theorem~6.2]{EllRobSan11Cone}, a \ca{} has strict comparison of positive elements by traces if and only if its Cuntz semigroup is almost unperforated and every quasitrace on $A$ is a trace.
(We refer to \cite{AntPerThi18TensorProdCu} and \cite{GarPer23arX:ModernCu} for details on the Cuntz semigroup.)

If $A$ is a simple \ca{} with real rank zero, stable rank one, and strict comparison of positive elements by traces, then $A$ necessarily admits nontrivial (lower semicontinuous, $[0,\infty]$-valued) traces, but there are possibly no bounded traces.
However, for any nonzero projections $p,q \in A$, the Choquet simplices of tracial states on $pAp$ and $qAq$ are canonically isomorphic. %, and they form a compact basis for the cone of traces on $A$.
We will say that $A$ has \emph{finitely many extremal traces} if for any nonzero projection $p \in A$, the Choquet simplex of tracial states on $pAp$ has finitely many extreme points.

%==========================================================================================
\begin{lma}
\label{prp:xrrHerMultTraces}
Let $A$ be a separable, simple \ca{} with real rank zero, stable rank one, strict comparison of positive elements by traces, and finitely many extremal traces (normalized at some nonzero projection).
Then:
\begin{enumerate}
\item
If $J \subseteq M(A)$ is a hereditary sub-\ca{}, then $\xrr(J) \leq 1$.
If we additionally assume that $K_1(A)=0$, then $\xrr(J)=0$.
\item
We have $\rr(K)=0$ and $\xrr(K) \leq 1$ for every hereditary sub-\ca{} $K \subseteq M(A)/A$.
\end{enumerate}
\end{lma}
\begin{proof}
For this proof, let us say that a \ca{} has property $(\ast\ast)$ if it satisfies that assumptions of this lemma, that is, if it is separable, simple, with real rank zero, stable rank one, strict comparison of positive elements by traces, and finitely many extremal traces (normalized at some nonzero projection).
In \cite{Ng22RR0PICorona}, a \ca{} is said to have property $(\ast)$ if it is nonunital, separable, simple, with strict comparison of positive elements by traces, projection injectivity and surjectivity, and quasicontinuous scale.

\smallskip

\textbf{Claim~1:} \textit{Every nonunital, nonelementary \ca{} with $(\ast\ast)$ has $(\ast)$.}
To prove the claim, we first note that every $\sigma$-unital, simple, nonunital, nonelementary \ca{} with real rank zero, stable rank one, and strict comparison of positive element by traces has projection injectivity and surjectivity by \cite[Theorem~4.5]{KafNgZha19PICorona}.
Essentially by definition (\cite[Definition~2.2]{KucPer11PICorona}), every simple \ca{} with real rank zero and with at most finitely many extremal quasitraces (normalized at some nonzero projection) has quasi-continuous scale.
Further, if a \ca{} has strict comparison of positive elements by traces, then every lower-semicontinuous quasitrace is a trace (\cite[Theorem~3.6]{NgRob16CommutatorsPureCa}).
The claim is proved by combining these results.

\smallskip

\textbf{Claim~2:} \textit{Let $B$ be a \ca{} with $(\ast)$, and let $K \subseteq M(B)/B$ be a hereditary sub-\ca{}.
Then $\rr(K)=0$ and $\xrr(K) \leq 1$.}
First, we deduce that $\rr(K)=0$ using that $\rr(M(B)/B)=0$ by \cite[Theorem~3.8]{Ng22RR0PICorona} and that real rank zero passes to hereditary sub-\ca{s} by \cite[Corollary~2.8]{BroPed91CAlgRR0}.
Further, $M(B)/B$ is purely infinite and has finitely many ideals by \cite[Theorem~6.11]{KafNgZha19PICorona};
see also \cite[Theorem~2.1]{Ng22RR0PICorona}.
This implies that $M(B)/B$ has finite primitive ideal space.
Since pure infiniteness passes to hereditary sub-\ca{s} by \cite[Proposition~4.17]{KirRor00PureInf}, and since the primitive ideal space of $K$ is naturally isomorphic to a subset of the primitive ideal space of $M(B)/B$, we obtain $\xrr(K) \leq 1$ by \autoref{prp:xrrHerPIFinIdl}.
This proves the claim.

\smallskip

To verify~(2), let $K \subseteq M(A)/A$ by a hereditary sub-\ca{}.
If $A$ is unital, then $M(A)/A=\{0\}$ and the result is clear.
If $A$ is elementary, then $M(A)/A$ is the Calkin algebra, and it follows that $K$ is purely infinite and simple, and therefore $\rr(K)=0$ by \cite[Proposition~V.3.2.12]{Bla06OpAlgs}, and $\xrr(K) \leq 1$ by \cite[Proposition~5.11(2)]{Thi23arX:RRExt}.
Finally, if $A$ is non-unital and non-elementary, then it follows from Claim~1 that $A$ has $(\ast)$, and the result follows from Claim~2.

\medskip

To verify~(1), let $J \subseteq M(A)$ be a hereditary sub-\ca.
Set $K := J/(J \cap A)$.
We have the following inclusion of extensions:
\[
\xymatrix@R-15pt{
0 \ar[r] 
& A \ar[r] 
& M(A) \ar[r] 
& M(A)/A \ar[r] 
& 0 \\
0 \ar[r] 
& J \cap A \ar[r] \ar@{}[u]|{\subseteqRotatedUp}
& J \ar[r] \ar@{}[u]|{\subseteqRotatedUp}
& K \ar[r] \ar@{}[u]|{\subseteqRotatedUp}
& 0 \\
}
\]

Since $K \subseteq M(A)/A$ is hereditary, we have $\rr(K)=0$ and $\xrr(K) \leq 1$ by~(2).
Further, since $J\cap A \subseteq A$ is hereditary, and since the properties entering the definition of $(\ast\ast)$ each pass to hereditary sub-\ca{s}, it follows that $J \cap A$ also has~$(\ast\ast)$.
In particular, $J \cap A$ is simple with real rank zero and stable rank one, and thus $\xrr(J \cap A) \leq 1$ by \cite[Proposition~5.11(1)]{Thi23arX:RRExt}.
Using \autoref{prp:xrrOfExtension} at the first step, we get
\[
\xrr(J) 
\leq \max\big\{ \xrr(J \cap A), \xrr(J/(J \cap A)) \big\}
\leq 1.
\]

Next, let us additionally assume that $K_1(A)=0$.
Since $J \cap A$ is Morita equivalent to $A$, we get $K_1(J \cap A)=0$.
Using that $J \cap A$ and $K$ both have real rank zero, we deduce that $\rr(J)=0$ by \cite[Proposition~4]{LinRor95ExtLimitCircle};
see also \cite[Proposition~2.5]{Thi23arX:RRExt}.
Further, we have $K_1(J)=0$ by \cite[Theorem~9]{Lin93ExpRank}.
Then $\xrr(J)=0$ by \autoref{prp:xrr0SufficientCond}.
\end{proof}

%==========================================================================================
In the next result, we do not assume that $A$ is separable or $\sigma$-unital.
The special case $A=\Cpct$ was already considered in \autoref{sec:ExtCpct}.

%==========================================================================================
\begin{prp}
\label{prp:rrMultExtSimpleSR1}
Let $0 \to A \to E \to B \to 0$ be an extension of \ca{s}.
Assume that $E$ is $\sigma$-unital, and that $A$ is simple, with real rank zero, stable rank one, strict comparison of positive elements by traces, and finitely many extremal traces (normalized at some nonzero projection).
Then:
\begin{enumerate}
\item
We have
\[
\rr(M(B)) \leq \rr(M(E)) \leq \max\big\{1, \rr(M(B)) \big\}.
\]
If we additionally assume that $K_1(A)=0$, then
\[
\rr(M(E)) = \rr(M(B)).
\]
\item
We have
\[
\rr(M(B)/B) \leq \rr(M(E)/E) \leq \max\big\{1, \rr(M(B)/B) \big\}.
\]
\end{enumerate}
\end{prp}
\begin{proof}
Pick $n$ such that $A$ has at most $n$ extremal traces (normalized at some nonzero projection).
The proof is similar to that of \autoref{prp:rrMultExtSimplePI}.
Let $\mathcal{F}$ denote the collection of separable sub-\ca{s} of $A$ that are simple, have real rank zero, stable rank one, strict comparison of positive elements by traces, and that have at most $n$ extremal traces (normalized at some nonzero projection).
Since each of the considered properties passes to inductive limits, we see that $\mathcal{F}$ is $\sigma$-complete.

To show that $\mathcal{F}$ is cofinal, let $B \subseteq A$ be a separable sub-\ca.
Since `real rank zero' satisfies the L{\"o}wenheim-Skolem condition (\cite[Paragraph~II.8.5.5]{Bla06OpAlgs}), we find a separable sub-\ca{} $C \subseteq A$ such that $B \subseteq C$ and $\rr(C)=0$.
Choose an increasing sequence of projections $(p_n)_n$ in $C$ that form an approximate unit.
Then the unital corners $p_nCp_n$ form an increasing sequence whose union is dense in $C$.

Consider $p_1Ap_1$, which is a unital, simple \ca{} with real rank zero, stable rank one, strict comparison of positive elements by traces, and at most $n$ extremal tracial states.
In \cite[Paragraph~II.8.5.5]{Bla06OpAlgs} it is shown that the property `unique tracial state' satisfies the L{\"o}wenheim-Skolem condition for unital \ca{s}.
The proof is easily adapted to show that `at most $n$ extremal tracial states' satisfies the L{\"o}wenheim-Skolem condition for unital \ca{s}.
Using also that the L{\"o}wenheim-Skolem condition is satisfied for the properties `simple' (\cite[Theorem~II.8.5.6]{Bla06OpAlgs}), for `real rank zero' and `stable rank one' (\cite[Paragraph~II.8.5.5]{Bla06OpAlgs}), and for `strict comparison of positive elements by traces' (\cite[Theorem~8.2.2]{FarHarLupRobTikVigWin21ModelThy}), we find a separable, unital, simple sub-\ca{} $D_1 \subseteq p_1Ap_1$ that has real rank zero, stable rank one, strict comparison of positive elements by traces, and at most $n$ extremal tracial states, and such that $p_1Bp_1 \subseteq D_1$.

Using the same argument successively in each $p_nAp_n$, we obtain a sequence $(D_n)_n$ such that $D_n$ is a separable, unital, simple sub-\ca{} of $p_nAp_n$ that has real rank zero, stable rank one, strict comparison of positive elements by traces, and at most $n$ extremal tracial states, and such that $D_{n-1} \subseteq D_n$ and $p_nCp_n \subseteq D_n$ for each $n \geq 2$.
Set $D:= \overline{\bigcup_n D_n}$.
Then $D$ belongs to $\mathcal{F}$ and contains $A$, as desired.
The situation is shown in the following diagram:
\[
\xymatrix@R-15pt{
p_1Ap_1 \ar@{}[r]|{\subseteq}
& p_2Ap_2 \ar@{}[r]|{\subseteq}
& p_3Ap_3 \ar@{}[r]|{\subseteq}
& \cdots \ar@{}[r]|{\subseteq}
& A \\
D_1 \ar@{}[r]|{\subseteq} \ar@{}[u]|{\subseteqRotatedUp}
& D_2 \ar@{}[r]|{\subseteq} \ar@{}[u]|{\subseteqRotatedUp}
& D_3 \ar@{}[r]|{\subseteq} \ar@{}[u]|{\subseteqRotatedUp}
& \cdots \ar@{}[r]|{\subseteq}
& D \ar@{}[r]|-{=} \ar@{}[u]|{\subseteqRotatedUp}
& \overline{\bigcup_n D_n} \\
p_1Cp_1 \ar@{}[r]|{\subseteq} \ar@{}[u]|{\subseteqRotatedUp}
& p_2Cp_2 \ar@{}[r]|{\subseteq} \ar@{}[u]|{\subseteqRotatedUp}
& p_3Cp_3 \ar@{}[r]|{\subseteq} \ar@{}[u]|{\subseteqRotatedUp}
& \cdots \ar@{}[r]|{\subseteq}
& C \ar@{}[r]|-{=} \ar@{}[u]|{\subseteqRotatedUp}
& \overline{\bigcup_n p_nCp_n}.
}
\]

Now the first estimate for $\rr(M(E))$ and the estimate for $\rr(M(E)/E)$ follow from combining \autoref{prp:rrMultUsingClub} with \autoref{prp:xrrHerMultTraces}, as in the proof of \autoref{prp:rrMultExtSimplePI}.
If we additionally assume that $K_1(A)=0$, then we use that `vanishing $K_1$-group' satisfies the L{\"o}wenheim-Skolem condition (\cite[Paragraph~II.8.5.5]{Bla06OpAlgs}) to obtain the improved result for $\rr(M(E))$ from \autoref{prp:rrMultUsingClub} and \autoref{prp:xrrHerMultTraces}.
\end{proof}

%==========================================================================================
Comparing \autoref{prp:rrMultExtSimplePI} with \autoref{prp:rrMultExtSimpleSR1} (or \autoref{prp:xrrHerMultPI} with \autoref{prp:xrrHerMultTraces}), the following question naturally arises:

%==========================================================================================
\begin{qst}
%\label{rmk:rrMultExtSimpleSR1}
Let $A$ be a separable, simple \ca{} with real rank zero, stable rank one, strict comparison of positive elements by traces, and finitely many extremal traces (normalized at some nonzero projection).
Assume that $K_0(A)=0$.
Let $K \subseteq M(A)/A$ be a hereditary sub-\ca{}.
Do we have $K_1(K)=0$?
Do we have $\xrr(K)=0$?
\end{qst}

%==========================================================================================
\begin{exa}
Consider an extension
\[
0 \to A \to E \to B \to 0
\]
where $A$ is a UHF-algebra (for example, the CAR algebra $M_{2^\infty}$), and $E$ is $\sigma$-unital.
Then
\[
\rr(M(E)) = \rr(M(B)).
\]
\end{exa}

%==========================================================================================
\begin{thm}
\label{prp:rrStableMultIIinf}
Let $N$ be a countably decomposable \IIinf~factor.
Then
\[
\xrr(N\otimes\Cpct) = 0, \andSep
\rr( M(N\otimes\Cpct) )
= \rr( M(N\otimes\Cpct)/N\otimes\Cpct ) 
= 1.
\]
\end{thm}
\begin{proof} 
Let $p\in N$ be a projection such that $N_0:=pNp$ is a type~\IIone~factor, and let $d\colon N_+\to[0,\infty]$ denote the unique dimension function on~$N$ with $d(p)=1$.
Since~$N$ is countably decomposable, it follows from \cite[Proposition~III.1.7.11]{Bla06OpAlgs} that $N$ contains a unique nontrivial (norm-closed) ideal $J$ given by
\[
J = \big\{ x\in N : d(x^*x)<\infty \big\}.
\]

Set $B := N/J$.
Since~$J$ is a maximal ideal in $N$, the quotient $B$ is simple.
Since~$N$ has real rank zero, so does $B$, and every projection in $B$ lifts to a projection in $N$.
For~$J$ contains every finite projection in $N$, we deduce that every (nonzero) projection in $B$ lifts to a properly infinite projection and is therefore properly infinite itself.
This implies that $B$ is simple and purely infinite.

Next, we show that $K_1(B) \neq 0$.
Consider the six-term exact sequence in $K$-theory induced by the extension $0 \to J \to N \to B \to 0$.
Since $J$ is Morita equivalent to the \IIone~factor $N_0$, we have $K_0(J) \cong K_0(N_0) \cong \mathbb{R}$.
Further, since $N$ is a \IIinf~factor, we have $K_0(N) = K_1(N) = 0$.
It follows that $K_1(B) \cong \mathbb{R}$.
\[
\xymatrix@R-15pt{
\mathbb{R} \ar@{}[r]|-{\cong} 
& K_0(J) \ar[r] 
& K_0(N) \ar[r] 
& K_0(B) \ar[dd] \\
& & 0 \ar@{}[u]|{\eqRotated} \ar@{}[d]|{\eqRotated}  \\
& K_1(B) \ar[uu] 
& K_1(N) \ar[l] 
& K_1(Q) \ar[l]
}
\]
Thus, by \autoref{prp:rrMultStableSimple}, we have
\[
\rr( M(B\otimes\Cpct) )
= 1.
\]

Now consider the extension
\[
0 \to J\otimes\Cpct \to N\otimes\Cpct \to B\otimes\Cpct \to 0.
\]
Then $N\otimes\Cpct$ is $\sigma$-unital.
Since $J\otimes\Cpct$ is Morita equivalent to the \IIone~factor $N_0$, we see that $J\otimes\Cpct$ is simple, has real rank zero, stable rank one, strict comparison of positive elements by traces, a unique trace normalized at $p$, and that $K_1(J\otimes\Cpct)=0$.
Therefore, \autoref{prp:rrMultExtSimpleSR1} applies and we obtain that
\[
\rr( M(N\otimes\Cpct) ) 
= \rr( M(B\otimes\Cpct) )
= 1,
\]

Arguing as in the proof of \autoref{prp:rrStableMultBdd}, we next deduce that $\xrr(N\otimes\Cpct)=0$.
Then, by applying \autoref{prp:EstimateWithXRR} for the extension
\[
0 \to N\otimes\Cpct \to M(N\otimes\Cpct) \to M(N\otimes\Cpct)/(N\otimes\Cpct) \to 0,
\]
we get
\[
\rr( M(N\otimes\Cpct)/(N\otimes\Cpct) ) 
= \rr( M(N\otimes\Cpct) )
= 1. \qedhere
\]
\end{proof}

%==========================================================================================
\begin{rmk}
\label{rmk:VNA}
We have computed the real ranks of stable multiplier and stable corona algebras of countably decomposable von Neumann factors.
What can we say about the analogous problem for general von Neumann algebras?

If $N$ is a finite von Neumann algebra, then $N$ has real rank zero, stable rank one, and trivial $K_1$-group, and then it follows from Lin's theorem \cite[Theorem~10]{Lin93ExpRank} that $M(N\otimes\Cpct)$ has real rank zero.
We expect that $M(N\otimes\Cpct)$ also has real rank zero if $N$ is a von Neumann algebra of type~\III.

On the other hand, if $N$ contains an ideal $I$ such that $N/I$ is a countably decomposable factor of type~\Iinf{} or type~\IIinf, then
\[
\rr( M(N\otimes\Cpct) ) 
\geq \rr( M((N/I)\otimes\Cpct) ) 
= 1.
\]
We suspect that this holds more generally whenever $N$ has a nonzero summand of type~\Iinf{} or type~\IIinf.
To complete the speculation, we expect that for every von Neumann algebra~$N$, we have
\[
\rr(M(N\otimes\Cpct)) = \begin{cases}
0, & \text{ if $N$ has no nonzero summands of type~\Iinf{} or type~\IIinf}, \\
1, & \text{ otherwise},
\end{cases}
\]
With view towards \autoref{prp:xrr0SufficientCond}, this would imply that $\xrr(N\otimes\Cpct)=0$ for every von Neumann algebra~$N$.
\end{rmk}

%\bibliographystyle{../../aomalphaMyShort}
%\bibliography{../../References}

\providecommand{\etalchar}[1]{$^{#1}$}
\providecommand{\bysame}{\leavevmode\hbox to3em{\hrulefill}\thinspace}
\providecommand{\noopsort}[1]{}
\providecommand{\mr}[1]{\href{http://www.ams.org/mathscinet-getitem?mr=#1}{MR~#1}}
\providecommand{\zbl}[1]{\href{http://www.zentralblatt-math.org/zmath/en/search/?q=an:#1}{Zbl~#1}}
\providecommand{\jfm}[1]{\href{http://www.emis.de/cgi-bin/JFM-item?#1}{JFM~#1}}
\providecommand{\arxiv}[1]{\href{http://www.arxiv.org/abs/#1}{arXiv~#1}}
\providecommand{\doi}[1]{\url{http://dx.doi.org/#1}}
\providecommand{\MR}{\relax\ifhmode\unskip\space\fi MR }
\providecommand{\MRhref}[2]{%
  \href{http://www.ams.org/mathscinet-getitem?mr=#1}{#2}
}
\providecommand{\href}[2]{#2}

\end{document}